\documentclass[11pt]{amsart}
\usepackage{amsfonts, bm}
\usepackage{mathrsfs}
\allowdisplaybreaks
\usepackage{hyperref}
\usepackage{color}
\usepackage{cite,enumerate}
\makeatletter

\newtheorem{theorem}{Theorem}[section]
\newtheorem{lemma}[theorem]{Lemma}

\newtheorem{proposition}[theorem]{Proposition}
\numberwithin{equation}{section}

\allowdisplaybreaks \numberwithin{equation}{section}

\begin{document}
\title[Precise propagation profile for   monostable free boundary problems]{Precise propagation profile for some monostable free boundary problems in time-periodic media}
\author[Y. Du, Z. Ma and Z.-C. Wang ]{Yihong Du$^\dag$, Zhuo Ma$^\dag$ and Zhi-Cheng Wang$^\ddag$}
 \thanks{
 \mbox{$^{\dag}$ School of Science and Technology, University of New England, Armidale, NSW 2351, Australia.} \\
 \mbox{\ \ \ \ \  $^{\ddag}$ School of Mathematics and Statistics, Lanzhou University, Lanzhou, Gansu 730000, China.} \\
\mbox{\ \ \ \ \ \ \    Emails:}  ydu@turing.une.edu.au (Y. Du),\ mazh25@foxmail.com (Z. Ma), \ wangzhch@lzu.edu.cn (Z.-C. Wang).
}
\date{\today}
\maketitle 
\begin{abstract}
  We consider  reaction-diffusion equations of the form
\begin{equation*}
u_t - d u_{xx} = f(t,u), \quad t>0,\ \ x \in [g(t), h(t)],
\end{equation*}
where $f(t,u)$ is periodic in $t$ and monostable in $u$, and the interval $[g(t), h(t)]$ represents the one dimensional population range of a species with density $u(t,x)$ at time $t$ and spatial location $x$.  The free boundaries $x=g(t)$ and $x=h(t)$ evolve subject to a ``preferred population density" condition at the habitat edges.
Analogous to the traveling wave solutions in the corresponding Cauchy problem, semi-wave solutions play a fundamental role in understanding the propagation phenomena governed by the  free boundary problem here.  But in contrast to the  Cauchy problem, where the KPP condition plays a subtle role in the precise approximation of its solution (with compactly supported initial function) by the traveling wave solution with minimal speed, here we  prove the existence and uniqueness of a semi-wave in a general monostable setting,  and  obtain a precise description of the convergence of the solution toward the semi-wave as time goes to infinity, where the KPP condition plays no special role. Previously,  such a sharp result was proved for a free boundary model only when $f$ is autonomous ($f=f(u)$, see \cite{D} or \cite{DL15} for a related free boundary model), or  a less precise result was obtained in the time-periodic case under an extra strong KPP condition on $f$ (see \cite{MDW}, or \cite{DGP} for a related free boundary model). This work appears to be the first to prove the sharp convergence result for a general monostable free boundary problem in a heterogeneous environment,
and we believe the methods developed here should have applications to related free boundary problems in heterogeneous media with nonlinearities more general than those of KPP type.

\bigskip
\textbf{Keywords}: Time-periodic media, free boundary,   sharp estimate, semi-wave solution

\medskip 
\textbf{AMS Subject Classification}: 35B40, 35K55, 35R35
\end{abstract}

\section{Introduction and main results}
The mathematical investigation of species expansion through reaction-diffusion models can be traced back to the classical works of Fisher \cite{Fisher} and Kolmogorov-Petrovskii-Piskunov (KPP) \cite{KPP}, who independently introduced what is now known as the Fisher-KPP equation
\begin{equation}\label{eq-kpp}
  u_t - d u_{xx} = f(u), \qquad  t>0,\ x\in\mathbb R
\end{equation}
with $f$ satisfying 
\[
f(0)=f(1)=0,~ f(u)>0 \text{ for } u\in(0,1), f(u)<0 \text{ for } u\in(1,\infty) \text{ and }f(u) \leq f'(0)u \text{ for } u>0.
\]
It was shown that \eqref{eq-kpp} admits traveling wave solutions of the form \( u(t,x) = \Phi(x - ct) \) if and only if \( c \geq c^* \), where \( c_* =2\sqrt{d\,f'(0)}\) is called the minimal speed. 
Moreover, \( c_* \) coincides with the asymptotic spreading speed of solutions with compactly supported initial data. These results laid the foundation for the modern theory of propagation in reaction-diffusion equations, including extensive works on a broad range of problems with various nonlinearities and boundary conditions (see, e.g., \cite{DL10, DL15,DN,S51, AW78, DM, zlatos, KMY, W82, fife77, DMZ,BLM,FLWW,CDLL,DHL}).

Classical autonomous models of the form \eqref{eq-kpp} provide a basic framework for studying reaction-diffusion phenomena. In the real world, however, periodic factors such as seasonal changes, diurnal cycles, and tidal rhythms modulate key population parameters, including diffusion and growth rates, carrying capacities, etc. Introducing temporal periodicity is therefore important for capturing realistic ecological variability and improving predictive accuracy.  Such temporal heterogeneity has duly  been taken into account in many works on propagation dynamics in reaction-diffusion systems (see, e.g., \cite{SLZ, DDL19, DGP,  Xin, ABC, FZ, wein02, shen11, N099,LLS,W16}). It should be noted that mathematically, such explicit time dependence gives rise to substantial analytical difficulties: the loss of autonomy renders classical techniques — such as phase-plane analysis and Lyapunov function methods — inapplicable, thereby compelling the development of new analytical tools.

 In this paper, we are concerned with  time-periodic free boundary problems of the form 
\begin{equation}\label{free-bound}
\begin{cases}
  u_t-du_{xx}=f(t,u), &t>0,~g(t)<x<h(t),\\
  u(t,g(t))=u(t,h(t))=\delta, &t>0,\\
  g'(t)=-\frac{d}{\delta}u_x(t,g(t)),&t>0,\\
  h'(t)=-\frac{d}{\delta}u_x(t,h(t)),&t>0,\\
  -g(0)=h(0)=h_0, u(0,x)=u_0(x), &-h_0<x<h_0,
  \end{cases}
\end{equation}
where $d>0$ and $h_0>0$ are given constants.  Such a model may be viewed as describing the expansion of some newly introduced or invasive species, whose territorial variation occurs without substantial biological or energetic trade-offs.  Here, $u(t,x)$ represents the population density at time $t$ and position $x$,  while the constant $\delta \in (0,1)$ represents a preferred population density of the species. 

The free boundaries $x = g(t)$ and $x = h(t)$ determine the moving fronts of the one dimensional population range $[g(t), h(t)]$, which evolve according to the equations in the second, third and fourth lines of \eqref{free-bound},
which arise from the biological assumption that the species maintains its preferred density $\delta$ at the range boundary by advancing or retreating the fronts (a detailed derivation  can be found in \cite{D}).  

The nonlinear function \( f:\mathbb{R}\times[0,\infty)\to\mathbb{R} \) is assumed to be
\begin{equation}\label{f-smooth}
\begin{cases}\mbox{\( C^{\alpha/2} \) in \( t\in\mathbb R \) uniformly with respect to \( u\ge0 \) for some \( \alpha\in(0,1) \), }\\
\mbox{\( C^1 \) in \( u\geq 0 \) uniformly with respect to \( t\in\mathbb{R} \).}
\end{cases}
\end{equation}
   Moreover, $f(t,u)$ is time-periodic in $t$ and monostable in $u$, namely:
\begin{enumerate}
  \item[($\rm{\bf{f_p}}$):] (Time-periodic) There exists $T>0$ such that $f(t+T,\cdot)=f(t,\cdot)$ for $t\in\mathbb{R}$;
  \item[($\rm{\bf{f_m}}$):](Monostable) For every fixed $t\in\mathbb R$, $f(t, 0)=f(t,1)=0$, $f(t, u)>0 $ for $u\in(0,1)$, $f(t, u)<0 $ for $u\in(1,+\infty)$, $\partial_uf(t, 0)>0>\partial_uf(t,1)$.
  \end{enumerate}
  
  The initial function $u_0(x)$ is always taken from the set 
 \begin{equation*}
  I(h_0):=\{\psi\in C^2([-h_0,h_0]):\psi(x)>0\text{ in }[-h_0,h_0], \ \psi(\pm h_0)=\delta\}.
 \end{equation*}

 When the nonlinear term $f$ in \eqref{free-bound} is time-independent, it was shown in \cite{D} and \cite{DLNN} that the unique solution $(u,g,h)$ exhibits successful spreading for all admissible initial data in both the monostable case with $\delta \in (0,1)$ and the bistable case (to reflect Allee effect) with $\delta \in (\theta_f^*,1)$, although the population range may temporarily shrink. Here for a bistable $f(u)$, $\theta_f^* \in (0,1)$ is the unique value determined by
\[\int_0^{\theta_f^*}f(s)ds=0,\quad \int_0^uf(s)ds>0 \text{ for }u\in(\theta_f^*,1].\]
In this latter case, the above conclusion indicates that the Allee effect can be suppressed if the species maintains its population density at an appropriate level along the range boundary.  

Moreover, the spreading proved in \cite{D,DLNN} has an asymptotic  speed and profile which can be  uniquely determined by the associated semi-wave solution.

In time-periodic media, under the assumptions ($\rm{\bf{f_p}}$) and  ($\rm{\bf{f_m}}$),    persistent spreading of  \eqref{free-bound} has been established in \cite{MDW}.  Furthermore, if additionally $f$ is of strong KPP type\footnote{This implies the usual KPP condition: $f(t,u)\leq f_u(t,0)u$ for $u\geq 0$.}, namely
\smallskip

  (${\bf{f_{kpp}}}$): $\frac{f(t,u)}{u}$ is nonincreasing in $u$ for all $t\in\mathbb R$,\smallskip
  \\
 then it was shown in \cite{MDW}  that  the associated semi-wave problem 
\begin{equation}\label{eqn-wave}
   \begin{cases}
     \Phi_t-d\Phi_{xx}+k(t)\Phi_x=f(t,\Phi),&t\in[0,T],~x\in(0,\infty),\\
\Phi(t,0)=\delta,&t\in[0,T],\\
\Phi(0,x)=\Phi(T,x), ~\Phi(t,\infty)=1, &t\in[0,T],~x\in[0,\infty)\\
   \end{cases}
  \end{equation}
admits a unique solution pair $(k(t), \Phi(t,x))$ satisfying 
  \begin{equation}\label{eq-phix1}
    \Phi_x(t,0)=\frac{\delta}{d}k(t)\quad  \text{ for }t\in[0,T],
  \end{equation}
  and the solution $(u,g,h)$ of \eqref{free-bound} has the following properties:
 \begin{equation*}
  \begin{cases}
   \lim_{t\to\infty}\frac{h(t)}{t}=\bar k^*:=\frac{1}{T}\int_{0}^{T}k^*(s)ds,\quad \lim_{t\to\infty}\frac{g(t)}{t}=-\bar k^*,\\
  \lim_{t\to\infty}\max_{|x|\leq (\bar{k}^*-\epsilon)t}|u(t,x)-1|=0 \text{ for any small }\epsilon>0.
  \end{cases}
\end{equation*}

For \eqref{free-bound} and related free boundary problems in heterogeneous media (see, e.g., \cite{DGP, DL}), it has been suspected that the KPP type conditions such as ${\bf (f_{kpp})}$ are only technical, but so far no one has been able to remove them in the theorems. 

In this paper we show that for \eqref{free-bound}, the extra condition ${\bf (f_{kpp})}$ indeed can be removed. Moreover, we considerably sharpen the above results on the asymptotic behavior of  $(u, g, h)$.
\smallskip

 Our main results are the following two theorems.

 \begin{theorem}\label{theo-semi-wave} Assume that  \eqref{f-smooth}, ${\bf (f_p)} $  and ${\bf (f_m)}$ hold.  Then for any given $\delta\in(0,1)$, there exists a unique positive, continuous and  $T$-periodic function  $k^*(t)=k_\delta^*(t)$  such that  problem \eqref{eqn-wave} has a  solution  $\Phi^*(t,x)$ satisfying \eqref{eq-phix1}. Moreover, such a solution $\Phi^*$ is unique and satisfies $\Phi_x^*(t,x)>0$ for $t\in[0,T]$ and  $x\in[0,\infty)$.
\end{theorem}

\begin{theorem}\label{thm-exact} Assume that   \eqref{f-smooth}, ${\bf (f_p)} $  and ${\bf (f_m)}$ hold. Let $(k^*,\Phi^*)$ be the semi-wave pair given in Theorem {\rm\ref{theo-semi-wave}} and $(u,g,h)$  the  solution of \eqref{free-bound} with $u_0\in I(h_0) $. Then there exist $g^*$, $h^*\in\mathbb{R}$ $($depending on $u_0)$ such that 
  \begin{eqnarray*}\begin{cases}
    \lim_{t\to\infty}[g(t)+\int_0^tk^*(s)ds]=g^*,~&\lim_{t\to\infty}[g'(t)+k^*(t)]=0,\\
    \lim_{t\to\infty}[h(t)-\int_0^tk^*(s)ds]=h^*,~&\lim_{t\to\infty}[h'(t)-k^*(t)]=0,
    \end{cases}
  \end{eqnarray*}
  \begin{eqnarray*}
    \begin{cases}\lim_{t\to\infty} \sup_{x\in[g(t),0]}|u(t,x)-\Phi^*(t,x-g(t))|=0,\\
    \lim_{t\to\infty} \sup_{x\in[0, h(t)]}|u(t,x)-\Phi^*(t,h(t)-x)|=0.
    \end{cases}
  \end{eqnarray*}
 \end{theorem}
 
 As mentioned above, the existence and uniqueness of a semi-wave pair to \eqref{eqn-wave} and \eqref{eq-phix1} were previously  established in \cite{MDW} under the additional assumption ${\bf (f_{kpp})}$. 
Theorem \ref{theo-semi-wave} above extends that result by removing this strong KPP condition and establishing its validity in the general monostable case.

 In  homogeneous environment, the corresponding result of Theorem \ref{theo-semi-wave} was proved in \cite{D} by extending a phase-plane approach used by Du and Lou~\cite{DL15} for a related but different free boundary problem.  For the time-periodic case, where the phase-plane technique is no longer applicable, a similar result was first proved by Du, Guo and Peng~\cite{DGP} for a model with the  free boundary condition of \cite{DL15}, and with a special nonlinear function $f(t,u)$ that satisfies the strong KPP condition ${\bf (f_{kpp})}$ automatically (and this property is crucial in the proofs there). One of the observations in \cite{DGP} is that the mean value of  $k(t)$ can be conveniently used for the classification and comparison of different semi-wave solutions, which will also be used in the analysis of this paper, but this is not enough to remove ${\bf (f_{kpp})}$ -- a rather new approach in the proof here is required (see the proof of Theorem \ref{theo-semi-wave} in Section 3 for details).

Theorem \ref{thm-exact} provides  a precise description of the asymptotic behavior of \eqref{free-bound} with a general monostable nonlinearity. Related results in homogeneous media and time-periodic media with a strong KPP nonlinearity can be found in \cite{D, DLNN, DMZ, DLZ15, DMZ2, SLZ,DGP} and the references therein. Theorem \ref{thm-exact} appears to be the first result of this nature to cover the general monostable case in a heterogeneous environment setting with free boundary. 

Since  ${\bf (f_{kpp})}$ is no longer required, the above results can be easily applied to some bistable cases. More precisely, suppose $f(t,u)$ satisfies \eqref{f-smooth} and ${\bf (f_p)}$, but instead of ${\bf (f_m)}$ it satisfies
\begin{enumerate}
  \item[($\rm{\bf{f_b}}$):](Bistable) $\begin{cases} \mbox{ For some $\theta\in (0,1)$ and every fixed $t\in\mathbb R$, $f(t, 0)=f(t,\theta)=f(t,1)=0$,}\\
  \mbox{ $f(t, u)>0 $ for $u\in(\theta,1)$, $f(t, u)<0 $ for $u\in (0, \theta)\cup (1,+\infty)$, }\\
  \mbox{ $\partial_uf(t, 0), \partial_uf(t,1)<0$ and $\underline f(u):=\inf_{t\in\mathbb R}f(t,u)$ satisfies $\int_0^1\underline f(u)du>0$.}
  \end{cases}$
  \end{enumerate}
  Then we have the following result.
  \begin{theorem}\label{thm-bistable} Suppose that  \eqref{f-smooth}, ${\bf (f_p)} $, ${\bf (f_b)}$ hold, and $\delta\in (\theta^*_{\underline f}, 1)$.  Then all the conclusions in Theorems \ref{theo-semi-wave} and \ref{thm-exact} remain valid.  
  \end{theorem}
  \begin{proof}
  We easily see that the conditions of Theorem 3.3 in \cite{DLNN} are satisfied. By the proof of this theorem in \cite{DLNN}, there exists $T_0>0$ such that
  \[
  u(t,x)\geq \delta \mbox{ for } t\geq T_0,\ x\in [g(t), h(t)].
  \]
  Since $\delta>\theta^*_{\underline f}>\theta$, we see that for $t\geq T_0$, the behaviour of $(u(t,x), g(t), h(t))$ is completely determined by $f(t,u)|_{\mathbb R\times [\theta^*_{\underline f}, \infty)}$. We may now modify $f(t,u)|_{\mathbb R\times [\theta, \theta^*_{\underline f}]}$,  so that the modified $f|_{\mathbb R\times [\theta, \infty)}$ satisfies  ${\bf (f_m)}$ with $0$  replaced by $\theta$. As $(u,g,h)$ is not affected by the modification for $t\geq T_0$, the conclusions in Theorems \ref{theo-semi-wave} and \ref{thm-exact} apply.
    \end{proof}
  
  Theorem \ref{thm-bistable} can be similarly extended to the combustion case, but we leave the details to the interested reader.

\medskip

To put this work into perspective against a more general background, let us mention some further related works.  
 As proved in \cite{MDW} (see also \cite{DGP,DG,D} for similar results for related free boundary models), the free boundary problem \eqref{free-bound} is closely related to the associated Cauchy problem
\begin{equation}\label{eq-cauchy}
  \begin{cases}
    U_t - dU_{xx} = f(t, U), & t > 0,~ x \in \mathbb{R}, \\
    U(0,x) = U_0(x),
  \end{cases}
\end{equation}
where
\begin{equation*}
  U_0(x) =
  \begin{cases}
    u_0(x), & x \in [-h_0, h_0], \\
    0, & x \in \mathbb{R} \setminus [-h_0, h_0].
  \end{cases}
\end{equation*}
To be precise, as \( \delta \to 0 \), the solution $(u^\delta, g^\delta, h^\delta)$ of \eqref{free-bound} satisfies  \( (g^\delta(t), h^\delta(t)) \to (-\infty, \infty) \) and \( u^\delta \to U \) locally uniformly in \( (t,x)\in(0,\infty)\times\mathbb{R} \). This conclusion indicates that the Cauchy problem \eqref{eq-cauchy} can be regarded as a limiting case of the free boundary problem when the boundary density level tends to zero. 

Nevertheless, the fine asymptotic behavior of the free boundary model differs substantially from that of the corresponding Cauchy problem, where the KPP condition plays a crucial role.  For example,
 a well-known ``logarithmic shifting term" appears in the wave approximation of the solution to the Cauchy problem, but this is not the case in the free boundary models (see Theorem \ref{thm-exact} for \eqref{free-bound} and \cite{DMZ} for a related but different free boundary model).

 More precisely,
in the homogeneous setting, when \( f \) satisfies the KPP condition, 
it is well known that there exist constants $C^+$ and $C^-$ such that the solution $U$ of the Cauchy problem with a compactly supported initial function satisfies
\[
\lim_{t \to +\infty} \sup_{x \in \mathbb R_{\pm}}
\left| U(t,x) - \Phi_{c_*}\!\Big( \pm x - c_* t + \frac{3}{c_*}\log t + C^{\pm} \Big) \right| = 0,
\]
where \( \Phi_{c_*} \) denotes the traveling wave with minimal speed \( c_* \); 
see \cite{Bramson,  HNRR, RRR} and references therein. It is well known that such a precise approximation result does not remain valid if the KPP condition is dropped.

In heterogeneous media, \eqref{eq-cauchy} and its many  variations have also been extensively investigated, with their spreading dynamics well characterized within the framework of generalised traveling wave solutions; see, for example, \cite{GF, wein02, HR11, LYZ,shen11,HR15,shen10} and the references therein.
In particular, for the spatial periodic version of \eqref{eq-cauchy}, the logarithmic shifting term was determined in \cite{HNRR16} under  the KPP  condition, namely, there exists  some bounded  function \( \xi : (0,\infty) \to \mathbb{R} \)  such that
\begin{equation*}
\lim_{t \to \infty} 
\left\| 
U(t,\cdot) - \Phi_{c_*}\!\left( t - \frac{3}{2c_* \lambda_*} \log t + \xi(t), \cdot \right)
\right\|_{L^\infty(\mathbb R_+)} 
= 0,
\end{equation*}
where \( \Phi_{c_*} \) is the pulsating traveling wave corresponding to the minimal speed \( c_* \), whose existence and uniqueness were established by
 Berestycki, Hamel and Roques \cite{BHR05}, and $\lambda^*>0$ is uniquely determined by $\Phi_{c_*}$. See also   \cite{G14, D15} for related results in asymptotically periodic or asymptotically homogeneous media. We further refer to \cite{HR15, N099, shen10,LYZ} and the references therein for results in the time-periodic setting. 

Let us also  briefly mention several related works concerning the asymptotic behavior of the Cauchy problem with non-compactly supported initial data. When the initial profile \( U_0 \) decays to zero at the same rate as a supercritical traveling wave $(c>c_*)$ as \( x \to \infty \), Hamel and Roques~\cite{HR11} established that, under the KPP condition, the corresponding solution converges to \( \Phi_c \) without any phase shift. 
If, instead, the initial datum lies between two supercritical traveling waves propagating at the same speed, it follows from~\cite{BMR} that the solution converges to the entire family of translates of this wave. 
Moreover, when both the initial data and its derivative decay exponentially, with a rate exceeding a certain critical threshold, Giletti~\cite{G14} proved the existence of a function \( m(t) \) satisfying \( m(t) = o(t) \) as \( t \to \infty \) such that
\begin{equation}\label{eq-up}
\lim_{t\to\infty} \|U(t,\cdot) - \Phi_{c^*}(t - m(t), \cdot)\|_{L^\infty(\mathbb{R}^+)} = 0.
\end{equation}
See also~\cite{ducrot,DM20,DM2,P20} and the references therein for related results involving more general nonlinearities.

Finally we would like to point out two significant differences of the current free boundary model \eqref{free-bound} from the one of Du and Lin \cite{DL10,DL15} and the Cauchy problem \eqref{eq-cauchy}. Firstly, the latter two models share the comparison principle that if two  solutions  $u(t,x)$ and $\tilde u(t,x)$ are  ordered at $t=0$, say $u(0,x)\leq \tilde u(0,x)$, then such an order is retained for all  later time $t>0$ (i.e., $u(t,x)\leq \tilde u(t,x)$). However, this is no longer the case for \eqref{free-bound}, and this difference creates considerable technical difficulties in the mathematical treatment of \eqref{free-bound}, although at the end, these three models share many common features in their long-time dynamics. Secondly, in the free boundary model
of \cite{DL10,DL15}, the population range $[g(t), h(t)]$ always expands as time increases, but this is no longer the case for \eqref{free-bound}, where the population range may shrink; see \cite{CDN, DLNN} for more details.\footnote{In the cases considered here, shrinking may occur before a certain time $t_0>0$, depending on the shape of the initial density $u_0(x)$; see Proposition \ref{prop-summ} (ii).}  
\medskip

The remainder of this paper is organized as follows.
Section~2 recalls several existing results, including comparison principles for free boundary problems and some basic properties of solutions to~\eqref{free-bound}, which form of foundation for the subsequent analysis. In Section~3, we prove the existence and uniqueness of semi-wave solutions without assuming ${\bf (f_{kpp})}$,
 and then determine its  asymptotic behavior as \(x \to \infty\). 
Section~4 is devoted to the proof of Theorem~\ref{thm-exact} based on  delicate constructions of upper and lower solutions, which enable us to obtain uniform bounds for \( |g(t)+\int_0^t k^*(s)\,ds| \) and \( |h(t)-\int_0^t k^*(s)\,ds| \), constituting a crucial step  in the proof. The convergence of the solution
 is then proved by first showing the convergence to some limiting functions along a special time sequence, made possible by the above bounds, and then proving the limiting functions are obtained from the semi-wave, which allows us to  revisit the construction of upper solutions to obtain a refined upper solution that leads to the desired convergence result. 

 \section{Preliminaries}

 \subsection{Some basic known results}
 For convenience, we collect here several basic facts  from \cite{MDW}, which will serve as the foundation for the analysis in the subsequent sections. We begin with two  comparison results.

 \begin{lemma}{\rm\cite[Lemma 2.2]{MDW}\label{lem-comp}}
  Assume that  $f\in C^1$ in $u$ uniformly with respect to $t\in \mathbb R$, $\tau\in(0,\infty)$ and   $(u,g,h)$ is the solution of \eqref{free-bound}. If  $\bar{g},\bar{h}\in C^1([0,\tau])$, $\bar{u}\in C(\bar{\Sigma}_\tau)\cap C^{1,2}( \Sigma_\tau)$ with $\Sigma_\tau=\{(t,x)\in\mathbb{R}^2:0<t\leq \tau,\bar{g}(t)< x<\bar{h}(t)\}$ satisfy 
 \begin{equation}\label{eq-com}
   \begin{cases}
     \bar{u}_t-d\bar{u}_{xx}\geq f(t,\bar{u}), &t\in (0, \tau],\ x\in ( \bar{g}(t), \bar{h}(t)),\\
    h(t)>\bar{g}(t)\geq g(t),   \bar{u}(t,\bar{g}(t))\geq u(t,\bar{g}(t)), &t\in (0,\tau],\\
    \bar{u}(t,\bar{h}(t))=\delta,~ \bar{h}'(t)\geq -\frac{d}{\delta}\bar{u}_x,&t\in (0, \tau],\ x=\bar{h}(t),\\
    \bar{u}(t,h(t))\geq u(t,h(t)), & t\in\{s\in (0, \tau]: h(s)\leq \bar{h}(s)\},\\
    h_0<\bar{h}(0),u_0(x)\leq,\not\equiv  \bar{u}(0,x),&x\in[\bar{g}(0),h_0].
   \end{cases}
 \end{equation}
 Then
 \begin{align*}
   &h(t)<\bar{h}(t)\text{ in }[0,\tau],\\
   &u(t,x)< \bar{u}(t,x)\text{ for }t\in(0,\tau] \text{ and }\bar{g}(t)<x<h(t).
 \end{align*}
 \end{lemma}
 \begin{lemma}{\rm\cite[Lemma 2.3]{MDW}}\label{lem-comp2}
  Assume that $f\in C^1$ in $u$ uniformly with respect to $t\in \mathbb R$, $\tau\in(0,\infty)$ and $(u,g,h)$ is the solution of \eqref{free-bound}.  If $\bar{g},\bar{h}\in C^1([0,\tau])$, $\bar{u}\in C(\bar{\Sigma}_\tau)\cap C^{1,2}( \Sigma_\tau)$ with $\Sigma_\tau=\{(t,x)\in\mathbb{R}^2:0<t\leq \tau,\bar{g}(t)< x<\bar{h}(t)\}$ satisfy
 \begin{equation*}
   \begin{cases}
     \bar{u}_t-d\bar{u}_{xx}\geq f(t,\bar{u}), &t\in (0,\tau],\ x\in (\bar{g}(t), \bar{h}(t)),\\
     \bar{u}=\delta, \bar{g}'(t)\leq -\frac{d}{\delta}\bar{u}_x,&t\in (0, \tau],\ x=\bar{h}(t),\\
    \bar{u}=\delta, \bar{h}'(t)\geq -\frac{d}{\delta}\bar{u}_x,&t\in (0, \tau], \ x=\bar{h}(t),\\
    \bar{u}(t,g(t))\geq u(t,g(t)), & t\in\{s\in (0, \tau]:  g(s)\geq \bar{g}(s)\},\\
    \bar{u}(t,h(t))\geq u(t,h(t)), & t\in \{s\in (0, \tau]:  h(s)\leq \bar{h}(s)\},\\
    [-h_0,h_0]\subset (\bar{g}(0),\ \bar{h}(0)),u_0(x)\leq,\not\equiv  \bar{u}(0,x),&x\in[-h_0,h_0].
   \end{cases}
 \end{equation*}
 Then
 \begin{align*}
   &[g(t),h(t)]\subset (\bar{g}(t),\bar{h}(t))\text{ for  }t\in[0,\tau],\\
   &u(t,x)< \bar{u}(t,x)\text{ for }t\in(0,\tau] \text{ and }g(t)<x<h(t).
 \end{align*}
 \end{lemma}

 In what follows, we refer to the triple \( (\bar{u}, \bar{g}, \bar{h}) \) given in Lemmas~\ref{lem-comp} and~\ref{lem-comp2} as an upper solution of \eqref{free-bound}. 
A lower solution is defined analogously by reversing these inequalities where appropriate.

We now recall several results concerning the qualitative properties of the unique solution to problem~\eqref{free-bound}.

 \begin{proposition}{\rm\cite{MDW}}\label{prop-summ}
  Assume that  $\rm{(\bf{f_p})}$ and   $\rm{(\bf{f_m})}$ hold, and   $(u,g,h)$ is the solution of \eqref{free-bound} with $u_0\in I(h_0)$. Then the following statements hold true:
  \begin{enumerate}[{\rm(i)}]
    \item There exists $M_0>0$ such that $0<u(t,x)\leq M_0$ and $|h'(t)|,|g'(t)|\leq M_0$ for $t\in[0,\infty)$ and $x\in[g(t),h(t)]$.
    \item There exists $T_0>0$ such that $-g'(t),h'(t)>0$ and $u(t,x)\geq \delta$ for $t\geq T_0$ and $x\in[g(t),h(t)]$.
    \item  As $t\to\infty$, 
    \begin{equation*}
     (g(t),h(t))\to\mathbb R \text{ and }u(t,x)\to 1 \text{ locally uniformly in } x\in\mathbb{R}.
        \end{equation*}
         \item $\limsup_{t\to\infty} u(t,x)\leq 1 \mbox{ uniformly for } x\in [g(t), h(t)].$
  \end{enumerate}
 \end{proposition}

\section{Existence, uniqueness and asymptotic behavior of the semi-wave}

In this section, we revisit the existence and uniqueness of semi-wave solutions in \cite{MDW} and extend the result there from the KPP case to a general monostable setting. The proof of uniqueness proceeds in two main steps: first, we show that equal integral averages of propagation speeds imply identical speeds and profiles; second, we deduce uniqueness by proving that such speed averages  are unique. In addition, we also investigate the asymptotic rate at which the semi-wave profile approaches the steady state~$1$. It is interesting to note that this convergence rate depends on $f_u(t,1)$, $d$, and the integral average of the wave speed, but is independent of~$\delta$. The design of our approach here is inspired by ~\cite{ABC,DGP}.

\begin{proof}[Proof of Theorem \ref{theo-semi-wave}]
From  \cite[Lemma 3.4]{MDW}, for any given $T$-periodic
nonnegative function $k(t)\in C^{\alpha/2}([0,T])$, problem \eqref{eqn-wave} admits a maximal positive  solution $\Phi^k$, which is strictly increasing in  $x$ and  strictly decreasing  with respect to $k(t)$ in the sense that $k_1(t)\leq, \not\equiv k_2(t) $ implies
\begin{eqnarray*}
  \Phi^{k_1}(t,x)>  \Phi^{k_2}(t,x),~   \Phi_x^{k_1}(t,0)>\Phi_x^{k_2}(t,0)\quad \text{ for }t\in[0,T], x\in(0,\infty).
\end{eqnarray*}
Our conclusion in Case (i) below (applied to the case $k_1=k_2=k$) indicates that $\Phi^k$ is indeed the unique  positive solution of \eqref{eqn-wave} for each given function $k(t)$. It then follows, by an argument analogous to that in the proof of \cite[Lemma~3.5]{MDW}, that for every $\delta\in (0,1)$ there exists a $T$-periodic function $k^*(t)\in C^{\alpha/2}([0,T])$ such that the pair $(k^*,\Phi^*)$ satisfies both \eqref{eqn-wave} and \eqref{eq-phix1}. In the following, we show that the pair $(k^*,\Phi^*)$  is  unique.

 Assume that both  $(k_1,\Phi^1)$ and $(k_2,\Phi^2)$ are solutions of  \eqref{eqn-wave}  and
 \begin{equation}\label{eq-phi12}
    \Phi^i_x(t,0)=\frac{\delta}{d}k_i(t)\quad  \text{ for }t\in[0,T],\ i=1,2.
 \end{equation}
Denote 
\begin{equation*}
  \lambda:=-\frac{1}{T}\int_0^Tf_u(s,1)ds=-\frac{1}{T}\int_{t-T}^{t}f_u(s,1)ds>0.
\end{equation*}
Since  $f_u$ is continuous in $u$ uniformly in $t$, there exists $\eta>0$ such that 
\begin{equation}\label{eq-fu}
      |f_u(t,u)-f_u(t,1)|\leq \frac{\lambda}{2}\quad \text{ for } t\in[0,T],~1-\eta\leq u\leq 1+\eta.
\end{equation}
 Define
\begin{equation*}
  \overline{k_i}:=\frac{1}{T}\int_{0}^{T}k_i(t)dt, ~~l_i(t):=\int_0^tk_i(s)ds \quad \text{ for }i=1,2
\end{equation*}
and 
\begin{equation*}
  l(t):=l_2(t)-l_1(t).
\end{equation*}
According to the relationship between $\overline{k_1}$ and $\overline{k_2}$, we divide the arguments into two cases.

\textbf{Case (i):} $\overline{k_1} =\overline{k_2}$. In this case, we show that $k_1\equiv k_2$ and $\Phi_1^*\equiv\Phi_2^*$.

Since $\overline{k_1} =\overline{k_2}$,  $l(t)$ is $T$-periodic in $t$ and $l(0)=l(T)=0$. Then there exists $t_0\in(0,T]$ such that 
\begin{equation*}
M_0:=\min_{t\in \mathbb R}l(t)=l(t_0+nT)\leq 0 \text{ for }n\in\mathbb Z.
\end{equation*}
 It follows that $ l(t)-M_0\geq 0$ for  $t\in\mathbb R$ and  $l(t_0+nT)-M_0=0$ for $n\in\mathbb N$. Thus, 
 \begin{equation}\label{eq-k12}
 l'(t_0+nT)=k_2(t_0+nT)-k_1(t_0+nT)=0 \text{ for }n\in\mathbb Z.
 \end{equation}
 Define  
 \begin{equation}\label{eq-rho}\begin{cases}
 \rho(t):=e^{\frac{\lambda }{2}(t-t_0+T)+\int_{t_0-T}^tf_u(s,1)ds} \text{ for }t\in[t_0-T,t_0],\\
  \sigma_0:=\displaystyle \frac{2\eta}{\|\rho\|_{L^\infty([t_0-T,t_0])}}.\end{cases}
 \end{equation}
 We note that $\rho(t_0-T)=1$ and $\sigma_0\leq 2\eta$.
 Since $\Phi^i(t,+\infty)=1$ uniformly in $t$, there exists $\xi_0>0$ large enough such that 
\begin{equation}\label{eq-phi-1}
|\Phi^i(t,x)-1|<\sigma_0/2 \quad \text{ for }t\in\mathbb R,~ x\geq\xi_0, ~i=1,2.
\end{equation} 
Moreover,  noting that $\delta<\Phi^i(t,x)<1$ for $t\in\mathbb R$ and $x>0$,  we can further choose $z_0>0$ sufficiently large such that   for every $ z\geq z_0$, it holds
\begin{equation}\label{eq-comp}
  \Phi^1(t,x)\leq \begin{cases}
    \Phi^2(t,x+l(t)-M_0+z), &t\in[t_0-T,t_0], ~x\in[0,\xi_0],\\
    \Phi^2(t,x+l(t)-M_0+z)+\sigma_0/2, &t\in[t_0-T,t_0],~ x\in(\xi_0,\infty).
  \end{cases}
\end{equation}
We will complete the proof in three steps.

\underline{Step 1.} We show that
\begin{equation}\label{sigmin}
  \sigma_{min}:=\inf\{\sigma>0:  \Phi^1(t_0-T,x)\leq \Phi^2(t_0-T,x+z)+\sigma \text{ for }x\geq 0, z\geq z_0\}=0.
\end{equation}

Thanks to \eqref{eq-comp},  we have $ \sigma_{min}\in[0,\sigma_0/2]$. We need to show that $\sigma_{min}=0$. 
For fixed $z\geq z_0$, let 
 \[\Phi^{2,z}(t,x):=\Phi^2(t,x+l(t)-M_0+z)+\sigma_{min}\rho(t).\]
 Then using \eqref{eq-rho} we obtain,
 by direct calculation,
\begin{align*}
  \mathcal{N}^1[\Phi^{2,z}]:&=\Phi^{2,z}_t-d\Phi^{2,z}_{xx}+k_1(t)\Phi^{2,z}_x-f(t,\Phi^{2,z})\\
  &=\sigma_{min} \rho'(t)+f(t,\Phi^2)-f(t,\Phi^{2,z})\nonumber \\
&=\Big[f_u(t,1)+\frac{\lambda}{2}-\int_0^1f_u(t,\Phi^{2}+\theta \sigma_{min} \rho   )d\theta\Big]\sigma_{min} \rho(t)\\
&\geq 0 \ \ \text{ for }t\in[t_0-T,t_0],\  x\in[\xi_0,\infty),\ 
\end{align*}
where, to reach the last inequality, we have used    \eqref{eq-fu} and the estimate
\begin{equation*}
 1-\eta \leq \Phi^{2}(t,x)+\theta(t,x) \sigma_{min} \rho(t)\leq 1+\frac{\sigma_0\|\rho\|_{L^\infty([t_0-T,t_0])}}{2}\leq 1+\eta \text{ for }t\in[t_0-T,t_0], x\in[\xi_0,\infty).
\end{equation*}
By  \eqref{eq-comp} and \eqref{sigmin}, we see that
\[
\Phi^1(t,x)\leq \Phi^{2,z}(t,x) \mbox{ on the parabolic boundary of } [t_0-T,t_0]\times(\xi_0,\infty).
\]
Therefore we can compare $\Phi^1(t,x)$ and $\Phi^{2,z}(t,x)$ over the region $[t_0-T,t_0]\times(\xi_0,\infty)$  to conclude, by the comparison principle, that
\begin{equation*}
  \Phi^1(t,x)\leq \Phi^{2,z}(t,x) \text{ for }t\in(t_0-T,t_0],\  x\in (\xi_0,\infty).
\end{equation*}
In particular, 
\begin{equation}\label{eq-psi12}
  \Phi^1(t_0,x)\leq\Phi^{2, z}(t_0,x)= \Phi^2(t_0,x+z)+\sigma_{min}\rho(t_0) \text{ for } x\in (\xi_0,\infty),\ z\geq z_0.
\end{equation}
Since  $\Phi^i(t,x), i=1,2, $ are periodic in $t$, it follows that
\begin{equation*}
  \Phi^1(t_0-T,x)\leq  \Phi^2(t_0-T,x+z)+\sigma_{min}\rho(t_0) \text{ for } x\in (\xi_0,\infty),  z\geq z_0.
\end{equation*}
This, together with \eqref{eq-comp} for $x\in[0,\xi_0]$ and the definition of  $\sigma_{min}$, implies that $\sigma_{min}\leq \sigma_{min} \rho(t_0)$. As $\rho(t_0)<1$, we must have $\sigma_{\min}=0$, as desired.
It follows that 
\begin{equation*}
  \Phi^1(t_0-T,x)\leq \Phi^2(t_0-T,x+z) \text{ for } x\geq 0,\ z\geq z_0.
\end{equation*}

\underline{Step 2.} We prove that
\begin{equation*}
  z_{min}:=\inf\{\tilde{z}>0:\Phi^1(t_0-T,x)\leq \Phi^2(t_0-T,x+z) \text{ for } x\geq 0,\ z\geq \tilde{z}\}=0.
\end{equation*}

Clearly, $z_{min}\in [0,z_0]$.   Suppose to the contrary that $z_{min}>0$.  Since   $\Phi^{2,z}(t,x)=\Phi^2(t,x+l(t)-M_0+z)$  satisfies, for each $z\geq z_{min}>0$,
\begin{equation*}
  \begin{cases}
    \Phi^{2,z}_t-d\Phi^{2,z}_{xx}+k_1(t)\Phi^{2,z}_x=f(t,\Phi^{2,z}), &t\in[t_0-T,t_0], x>M_0-l(t)-z, \\
    \Phi^{2,z}(t,0)\geq \Phi^2(t,z)> \delta, &t\in[t_0-T,t_0],
  \end{cases}
\end{equation*}
one deduces from the strong maximum principle and the definition of $z_{min}$ that 
\begin{equation*}
  \Phi^1(t,x)<\Phi^{2,z}(t,x)= \Phi^2(t,x+l(t)-M_0+z)~~ \text{ for }t\in(t_0-T,t_0],\ x\geq 0,\ z\geq z_{min}.
\end{equation*} 
This implies, by the continuity of $\Phi^2$ and the fact that $\Phi^i(t,\infty)=1$ uniformly in $t\in\mathbb R$, that   \eqref{eq-comp} holds  for every $z\geq z_{min}-\epsilon$ provided that $\epsilon>0$ is small enough. We may now repeat the arguments in  Step 1 to obtain
\begin{equation*}
  \Phi^1(t_0-T,x)\leq \Phi^2(t_0-T,x+z)~~ \text{ for } x\geq 0,\ z\geq z_{min}-\epsilon.
\end{equation*}
But this contradicts the definition of $z_{min}$. Hence $z_{min}=0$.

\underline{Step 3.} We show that  $k_1\equiv k_2$ and $\Phi^1\equiv \Phi^2$.

By Step 2 and $l(t)-M_0\geq 0$  we obtain 
\begin{equation*}\begin{cases}
  \Phi^1(t_0-T,x)\leq \Phi^2(t_0-T,x)=\Phi^2(t_0-T, x+l(t_0-T)-M_0)=\Phi^{2,0}(t_0-T,x) \text{ for } x\geq 0,\\
  \Phi^1(t,0)=\delta=\Phi^2(t,0)\leq \Phi^2(t,l(t)-M_0)=\Phi^{2,0}(t,0) \text{ for }t\in[t_0-T,t_0].
  \end{cases}
\end{equation*}
Therefore, unless 
\begin{equation}\label{equiv}
 \begin{cases} \Phi^1(t_0-T,x)\equiv \Phi^2(t_0-T,x) \text{ for } x\geq 0,\\
 l(t)\equiv M_0 \mbox{ for } t\in\mathbb R,
 \end{cases}
\end{equation}
we can apply the strong maximum principle to $\Phi^1(t,x)-\Phi^{2,0}(t,x)$ over $(t_0-T,t_0]\times[0,\infty)$ to deduce
\begin{equation*}
  \Phi^1(t,x)< \Phi^{2,0}(t,x) \text{ for }t\in(t_0-T,t_0],\ x\in(0,\infty).
\end{equation*}
Since  
\begin{equation*}
 \Phi^1(t_0,0)=\delta=\Phi^2(t_0,0)=\Phi^{2,0}(t_0,0), 
\end{equation*}
the above inequality and the Hopf boundary lemma imply that 
\begin{equation*}
  \Phi^1_x(t_0,0)<\Phi^{2,0}_x(t_0,0)=\Phi^2_x(t_0,0),
\end{equation*}
which in turn implies $k_1(t_0)<k_2(t_0)$ by \eqref{eq-phi12}, so we reach a contradiction to \eqref{eq-k12}. Hence \eqref{equiv} holds, which leads to  $k_1\equiv k_2$ and $\Phi^1\equiv \Phi^2$, as desired.

\textbf{Case (ii):} $\overline{k_1} \neq \overline{k_2}$. In this case, we derive a contradiction, which means  this case can never happen.

Without loss of generality, we assume that $\overline{k_1}>\overline{k_2}$, as the other  case can be argued symmetrically. Then 
\begin{equation*}
    \lim_{t\to-\infty}\frac{l_1(t)}{t}=\overline{k_1}>\overline{k_2} =\lim_{t\to-\infty}\frac{l_2(t)}{t}.
\end{equation*}
Noting  $l_1(0)=l_2(0)=0$, there exists $t_0\leq 0$ such that 
\begin{equation*}
l(t):=l_2(t)-l_1(t)>0 \text{ for }t<t_0,  \quad   l(t_0)=0,   
\end{equation*}
and therefore
\begin{equation}\label{eq-k122}
  l'(t_0)=k_2(t_0)-k_1(t_0)\leq 0.
\end{equation}

Similarly to Case (i),  we can choose $z_0$ large enough such that,   for each $ z\geq z_0 $,
\begin{equation}\label{eq-ps12}
  \Phi^1(t,x)\leq \begin{cases}
    \Phi^2(t,x+l(t)+z), &t\in[t_0-T,t_0],\  x\in[0,\xi_0],\\
    \Phi^2(t,x+l(t)+z)+\sigma_0/2, &t\in[t_0-T,t_0],\ x\in(\xi_0, +\infty),
  \end{cases}
\end{equation}
with  $\sigma_0$ and $\xi_0$  given by \eqref{eq-rho} and \eqref{eq-phi-1}, respectively.
Define 
\begin{equation*}
  \tilde{\sigma}_{min}:=\inf\{\sigma>0:  \Phi^1(t_0-T,x)\leq \Phi^2(t_0-T,x+l(t_0-T)+z)+\sigma \text{ for }x\geq 0, z\geq z_0\}.
\end{equation*}
Then $\tilde{\sigma}_{min}\in[0,\sigma_0/2]$.
Let $\rho(t)$ be  given by \eqref{eq-rho}. By an argument analogous to Step~1 of Case~(i) and comparing  $\Phi^1(t,x)$ and $\Phi^2(t,x+l(t)+z)+\tilde{\sigma}_{min}\rho(t)$ over the region   $[t_0-T,t_0]\times[\xi_0,\infty)$, we obtain 
 \begin{equation*}
  \Phi^1(t_0,x)\leq \Phi^2(t_0,x+l(t_0)+z)+\tilde{\sigma}_{min}\rho(t_0)  \text{ for } x\in [\xi_0,\infty),\ z\geq z_0.
\end{equation*}
Since $l(t_0-T)>0=l(t_0)$ and $\Phi^2$ is increasing in $z$,  it follows that 
  \begin{equation*}
  \Phi^1(t_0,x)\leq \Phi^2(t_0,x+l(t_0-T)+z)+\tilde{\sigma}_{min}\rho(t_0)   \text{ for } x\in [\xi_0,\infty),\ z\geq z_0.
\end{equation*}
By  the periodicity of $\Phi^{i}$ and  \eqref{eq-ps12}, we further obtain
 \begin{equation*}
  \Phi^1(t_0-T,x)\leq \Phi^2(t_0-T,x+l(t_0-T)+z)+\tilde{\sigma}_{min} \rho(t_0)  \text{ for } x\in [0,\infty),\ z\geq z_0.
\end{equation*}
From the definition of $\tilde{\sigma}_{min}$, we conclude that $\tilde{\sigma}_{min}\leq \tilde{\sigma}_{min}\rho(t_0)$. As $\rho(t_0)<1$, this implies $\tilde{\sigma}_{min}=0$.

Then, we can argue as in Step 2 of Case (i) to deduce
 \begin{equation}\label{eq-piz12}
  \Phi^1(t_0-T,x)\leq \Phi^2(t_0-T,x+l(t_0-T)+z) \text{ for } x\in [0,\infty),\ z\geq 0.
\end{equation}
Applying the strong maximum principle and Hopf boundary lemma to $\Phi^1(t,x)$ and  $\Phi^2(t,x+l(t))$ over $(t_0-T,t_0]\times [0,\infty)$, we  deduce that 
 \begin{equation*}
  \Phi^1(t,x)< \Phi^2(t,x+l(t)), ~~ \Phi^1_x(t_0,0)<\Phi^2_x(t_0,0) \text{ for }t\in(t_0-T,t_0],\ x\in (0,\infty).
\end{equation*}
This leads to, by \eqref{eq-phi12}, 
\begin{equation*}
  k_1(t_0)<k_2(t_0),
\end{equation*}
 which contradicts \eqref{eq-k122}. The proof is thus complete.
\end{proof}

\begin{lemma}\label{lem-1-phi}
 Let $(k^*(t),\Phi^*(t,x))$ be the unique semi-wave pair in Theorem \ref{theo-semi-wave}, and
  \[ \overline{k^*}:=\frac{1}{T}\int_0^Tk^*(s)ds,~~a_0:=-\frac{1}{T}\int_0^Tf_u(s,1)ds,\ \mu_0:=\frac{-\overline{k^*}+\sqrt{\overline{k^*}^2+4a_0d}}{2d}.\]
  Then for any $\mu\in (0,\mu_0)$, there exists $N_1=N_1(\mu)>0$ such that
  \begin{equation}\label{phi-infty}
    |1-\Phi^*(t,x)|+|\Phi_x^*(t,x)|+|\Phi_{xx}^*(t,x)|+|\Phi_t^*(t,x)|\leq N_1 e^{-\mu x} ~\text{  for  }t\in[0,T],~x\in[0,\infty).
  \end{equation}
 
\end{lemma}
\begin{proof}
Define
\[
\begin{cases}
\phi(t):=e^{-a_0t-\int_0^tf_u(s,1)ds} &\text{ for }t\geq0,\\
U(t,x):=(1-\Phi^*(t,x+H))\phi(t)&\text{ for } t\geq 0,\ x\geq 0,
\end{cases}\]
with the constant $H>0$  to be determined later.  Clearly, $\phi$ and $U$ are $T$-periodic in $t$. 
Direct calculations give 
\begin{eqnarray*}
  &&U_t(t,x)=-\Phi^*_t(t,x+H)\phi(t)-[a_0+f_u(t,1)]U(t,x);\\
  &&U_x(t,x)=-\Phi^*_x(t,x+H)\phi(t), ~~U_{xx}(t,x)=-\Phi^*_{xx}(t,x+H)\phi(t).
\end{eqnarray*}
Denote 
\[C(t,x)=C_H(t,x):=\frac{f(t,\Phi^*(t,x+H))+f_u(t,1)(1-\Phi^*(t,x+H))}{1-\Phi^*(t,x+H)}.\]
Since $f(t,u)$ is $C^1$ in $u$, we have  
\begin{equation*}
  f(t,\Phi^*)+f_u(t,1)(1-\Phi^*)=o(|\Phi^*-1|) \ \text{ as } |\Phi^*-1|\to 0.
\end{equation*}
For any fixed $\mu\in(0,\mu_0)$, let $a:=d\mu^2+\overline{k^*}\mu$.
Then  $a\in(0,a_0)$  and there exists $H>0$ large enough such that  
\begin{eqnarray*}
  |C(t,x)|\leq a_0-a \ \text{ for }t\in[0,T],\ x\in[0,\infty).
\end{eqnarray*}
Moreover, 
\begin{equation}\label{eq-Lu}\begin{aligned}
  \mathcal{L}[U]&:=U_t-dU_{xx}+k^*(t)U_x+a_0 U+C(t,x)U \\
  &= -\phi[\Phi^*_t-d\Phi^*_{xx}+k^*(t)\Phi^*_x]-f_u(t,1)U+\frac{f(t,\Phi^*)+f_u(t,1)(1-\Phi^*)}{1-\Phi^*}U  \\
  &= -\phi f(t,\Phi^*)-f_u(t,1)U+[f(t,\Phi^*)+f_u(t,1)(1-\Phi^*)]\phi\\
  &=0\ \  \mbox{ for } t\in [0, T],\ x\geq 0.
  \end{aligned}
\end{equation}

Let us now estimate $U$ by constructing an appropriate upper solution. Set 
\[M:= \max_{t\in[0,T]}\left|\overline{k^*}t-\int_0^tk^*(s)ds\right|.\]
For fixed $\mu\in(0,\mu_0)$,  $q\in[0,\infty)$, and variables $t\in[0,T], \ x\in[0,\infty)$,  define the function $\bar U(t,x)=\bar U_q(t,x)$ by
\[\bar{U}_q(t,x):= e^{-\mu[x+M+\overline{k^*}t-\int_0^tk^*(s)ds]}+ q.\]
By direct calculation we have
\begin{eqnarray*}
 && \bar{U}_t= -\mu  [\overline{k^*}-k^*(t)]e^{-\mu[x+M+\overline{k^*}t-\int_0^tk^*(s)ds]};\\
&&\bar{U}_x=-\mu   e^{-\mu[x+M+\overline{k^*}t-\int_0^tk^*(s)ds]};\\
&& \bar{U}_{xx}= \mu^2  e^{-\mu[x+M+\overline{k^*}t-\int_0^tk^*(s)ds]}.
\end{eqnarray*}
Hence, 
\begin{equation}\label{bar-U}
\begin{aligned}
  \mathcal{L}[\bar{U}]=&\ \bar{U}_t-d\bar{U}_{xx}+k^*(t)\bar{U}_x+a_0 \bar{U}+C(t,x)\bar{U} \\
  =&\ e^{-\mu[x+M+\overline{k^*}t-\int_0^tk^*(s)ds]}\Big[-\mu(\overline{k^*}-k^*(t))-d\mu^2-k^*(t)\mu +a_0+C\Big]+(a_0+C)q\\
  =&\ e^{-\mu[x+M+\overline{k^*}t-\int_0^tk^*(s)ds]}\big[-a+a_0+C\big]+(a_0+C)q\\
  \geq&\  0 \ \  \mbox{ for } t\in [0, T],\ x\geq 0.
\end{aligned}
\end{equation}

Denote
\[m_0:=\min_{t\in[0,T]} e^{-\mu(M+\overline{k^*}t-\int_0^tk^*(s)ds)}. \]
Then $m_0\in(0,1]$. Since $\Phi^*(t,\infty)=1$ uniformly in $t\in\mathbb R$, by increasing $H$ if necessary, we have
\[U(t,x)< m_0   \ \text{ for }t\in[0,T],~ x\in[0,\infty).\] Thus, for any fixed $q\geq m_0$, we have
\[U(t,x)\leq \bar{U}_q(t,x)\text{ for  }t\in[0,T],\ x\in[0,\infty).\]
Therefore, the constant 
\[q_0:=\inf\left\{q\in\left(0,\infty\right):U(0,x)\leq \bar{U}_q(0,x)\text{ for  }x\in[0,\infty)\right\}\]
is well-defined and 
\[U(0,x)\leq \bar{U}_{q_0}(0,x)\text{ for  } x\in[0,\infty).\]
  We claim that $q_0=0$.  Assume to the contrary that $q_0>0$. Then for $t\in[0,T]$ we have
\[\bar{U}_{q_0}(t,0)>e^{-\mu[M+\overline{k^*}t-\int_0^tk^*(s)ds]}\geq   m_0\geq U(t,0).\]
The above inequalities at the parabolic boundaries of the region $[0,T]\times [0,\infty)$, together with \eqref{eq-Lu} and \eqref{bar-U}, allow us to use the strong maximum principle to conclude that 
\[U(t,x)< \bar{U}_{q_0}(t,x) \text{ for }t\in(0,T],\ x\in[0,\infty),\]
which leads to
\[U(0,x)=U(T,x)<\bar{U}_{q_0}(T,x)=\bar{U}_{q_0}(0,x) \text{ for }  x\in[0,\infty).\]
Combining this with 
\[U(0,\infty)=0<q_0=\bar{U}_{q_0}(0,\infty),\]
we easily see  that
\[U(0,x)\leq \bar{U}_{q_0-\epsilon}(0,x)\text{ for  }x\in[0,\infty) \mbox{ and some small } \epsilon>0.\]
This contradicts the definition of $q_0$. Hence, $q_0=0$ and we can again apply the  comparison principle  to $U$ and $\bar{U}_0$  over  $[0,T]\times[0,\infty)$ to deduce
\begin{eqnarray*}
  U(t,x)\leq \bar{U}_0(t,x)= e^{-\mu[x+M+\overline{k^*}t-\int_0^tk^*(s)ds]} \leq  e^{-\mu x} \  \text{ for }t\in(0,T],\ x\in[0,\infty).
\end{eqnarray*}
Since $\phi(t)$ is a positive periodic function, this implies, by the definition of $U$, 
\[0\leq 1-\Phi^*(t,x)\leq C_1e^{-\mu x} \text{ for all } t\in[0,T],~x\in[0,\infty) \mbox{ and some } C_1>0.\] 
Because $w:=\Phi^*-1$ satisfies 
\[\begin{cases}
w_t-dw_{xx}+k(t)w_x=f(t,\Phi^*)-f(t,1)=f_u(t,\theta(t,x))w& \mbox{ for } \ t\in\mathbb R, \ x>0,\\
w=\delta-1& \mbox{ for }\ t\in\mathbb R,\ x=0,
\end{cases}\]
with $\theta(t,x)\in[\Phi^*(t,x),1]$,  it follows from the standard parabolic $L^p$ estimates and Sobolev embedding theorems, that  there exists $N_1>0$ such that, for any $y\in[0,\infty)$,
\[\|w\|_{C^{1, 2}([T,2T]\times[y,\ y+1])}\leq C_2 \|w\|_{L^{\infty}(\mathbb R\times [\max\{0,\, y-1\},\ y+2])}\leq N_1e^{-\mu y},\]
which implies the desired result.
\end{proof}

\section{Proof of Theorem \ref{thm-exact}}

In this section, we determine the precise asymptotic spreading speed and propagation profile for the solution of \eqref{free-bound}. 
By carefully constructing suitable upper and lower solutions, following the style of \cite{fife77}, we  derive a uniform bound for
\(|g(t) + \int_0^t k^*(s)\,ds|\) and \(|h(t) - \int_0^t k^*(s)\,ds|\) in Subsections~4.1 and~4.2.  
Making use of these bounds, we then prove  in Subsection~4.3 that along a sequence of time $t_n\to\infty$, the solution  \(u\) under suitable translations in both \(t\) and \(x\) converges to a certain limit function. 
Finally, in Subsections~4.4 and~4.5, through some careful analysis based on various estimates, we show that the limit function obtained in Subsection 4.3  coincides with  a spatial translation and reflection of the semi-wave, which enables us to refine the upper solutions constructed in Subsection 4.1 to deduce the desired convergence result. The strategy of our approach here is inspired by  \cite{DMZ,DMZ2, D}.

\subsection{Upper bound}
In view of $\Phi^{*}\in  C^{1+\alpha/2,2+\alpha}([0,T]\times[0,\infty))$, we have
\[k^*(\cdot)=\frac{d}{\delta}\Phi^{*}_x (\cdot,0)\in C^{1+\alpha/2}([0,T]).\]
Let
\begin{eqnarray*}
  \begin{cases} M:= \displaystyle \frac{1}{2d}\left[\max_{(t,u)\in[0,T]\times[0,\delta]}f(t,u)+\frac{\delta }{d}\|k^*\|_{C^{1}([0,T])}\right],\\[2mm]
   \Upsilon(t,x):=-Mx^2+\frac{\delta}{d}k^*(t)x+\delta.
   \end{cases}
\end{eqnarray*}
Clearly $\Upsilon(t, 0)=\delta>0$.
Let  $a\in(0,1)$ be any fixed number small enough such that  $\Upsilon(t,x)>0$ for $(t,x)\in[0,T]\times[-a,0]$ and  
\begin{equation*}
  \delta_a:= \max_{t\in[0,T]}\Upsilon(t,-a).
\end{equation*} 
Define 
\begin{equation*}
  \bar{\Phi}(t,x):=\begin{cases}
  \Upsilon(t,x)& \text{ for }-a\leq x\leq 0,\\
  \Phi^*(t,x) &\text{ for }x\geq 0.
  \end{cases}
\end{equation*}
Clearly, $\delta_a<\delta$ and $\bar{\Phi}\in C^1([0,T]\times[-a,\infty))$ is strictly increasing in $x$. Denote
\begin{equation}\label{eq-k+-}
   k_{min}:=\min_{t\in[0,T]}k^*(t),\  k_{max}:=\max_{t\in[0,T]}k^*(t).
\end{equation}
For $t\in [0,T],\ x\in[-a,0]$, direct computations give
\begin{equation}\label{eq-upes}
  \begin{cases}
    0<\frac{\delta}{d} k_{min} \leq  \bar{\Phi}_x(t,x)=\Upsilon_x(t,x)=-2Mx+\frac{\delta}{d}k^*(t)\leq2M a+\frac{\delta}{d}k_{max}; \\
|\bar{\Phi}_t(t,x)|=|\Upsilon_t(t,x)|=\frac{\delta}{d}|k^{*}(t)'x|\leq \frac{\delta}{d}\|k^*\|_{C^1([0,T])}|x|.
  \end{cases}
\end{equation}

Due to $f_u(t,1)<0$  and $f\in C^1$ in $u$ uniformly with respect to $t$, there exist $\sigma_0>0$  and $\eta>0$ such that 
\begin{equation}\label{eq-f_usig}
  f_u(t,u)\leq -\sigma_0 \text{ for }t\in[0,T],\ u\in[1-\eta,1+\eta].
\end{equation}
Fix $q_0\in(0,\min\{\eta, \delta-\delta_a \})$. Since $\Phi^*(t,\infty)=1$ uniformly in $t\in\mathbb{R}$, we can find  $K_0\gg 1$  such that 
\begin{equation}\label{eq-1q_0}
    \Phi^*(t,x)\geq 1-\frac{q_0}{2}\geq  1-\eta ~\text{ for } t\in\mathbb{R}, ~x\geq K_0.
\end{equation}
By Proposition \ref{prop-summ} (iv),  there exists  $\tau^*>0$  such that
\begin{equation*}
  -g(t),h(t)>0 \text{ and }u(t,x)\leq 1+\frac{q_0}{2}~ \text{ for } t\geq \tau^*,\ x\in[g(t),h(t)].
\end{equation*}

Now, we define
\begin{equation}\label{eq-ups}
  \begin{cases}
    \displaystyle\bar{h}(t):=\int_{\tau^*}^{t}k^*(s)ds+\gamma(1-e^{-\sigma_0(t-\tau^{*})})+h(\tau^*)+a+K_0 &\text{ for } t>\tau^*,\\[2mm]
    \bar{u}(t,x):=\bar{\Phi}(t, \bar{h}(t)-x+z_\delta(t))+ q_0e^{-\sigma_0(t-\tau^{*})}&\text{ for }t>\tau^*,\ x\in[g(t),\bar{h}(t)],
  \end{cases}
\end{equation}
where   $\gamma>0$ is a  constant to be determined later, and  the function $z_\delta(t)$ is determined by  
\begin{equation}\label{zdelta}
\bar{\Phi}(t,z_\delta(t))+q_0e^{-\sigma_0(t-\tau^{*})}=\delta.
\end{equation}
In view of the restriction on $q_0$, we know that  $z_\delta(t)\in(-a,0)$ is well-defined  for $t\geq\tau^*$. Moreover, $z_\delta(t)\to 0$ as $t\to\infty$ and 
\begin{equation}
  -q_0e^{-\sigma_0(t-\tau^{*})}=\bar{\Phi}(t,z_\delta(t))-\bar{\Phi}(t,0)= \bar{\Phi}_x(t,\theta(t))z_\delta(t)
\end{equation}
for some  $\theta(t)\in[z_\delta(t),0]$. It then  follows that  
\begin{equation}\label{eq-zdt}
  |z_\delta(t)|\leq B_0q_0e^{-\sigma_0(t-\tau^*)}\text{ for } t\geq \tau^*, 
\end{equation}
where 
\begin{equation*}
  B_0:=\frac{1 }{\min_{(t,x)\in[0,T]\times[-a,0]}\bar{\Phi}_x(t,x)}=\frac{d }{\delta k_{min}}.
\end{equation*}
Differentiating \eqref{zdelta} we obtain
\begin{equation*}
  \bar{\Phi}_x(t,z_\delta(t))z_\delta'(t)= -\bar{\Phi}_t(t,z_\delta(t))+ \sigma_0q_0 e^{-\sigma_0(t-\tau^*)}.
  \end{equation*}
This, together with \eqref{eq-upes} and  $z_\delta(t)\in(-a,0)$, implies that 
\begin{eqnarray*}
  \frac{\delta}{d} k_{min}|z_\delta'(t)|\leq|\bar{\Phi}_x(t,z_\delta(t)) z_\delta'(t)|\leq \frac{\delta}{d}\|k^*\|_{C^1([0,T])}|z_\delta(t)|+ \sigma_0q_0 e^{-\sigma_0(t-\tau^*)}.
\end{eqnarray*}
Therefore, by  \eqref{eq-zdt} we obtain
\begin{equation}\label{eq-zd'}
\begin{aligned}
  |z_\delta'(t)|&\leq  \frac{d}{\delta  k_{min}}\left( \frac{\delta}{d}\|k^*\|_{C^1([0,T])}B_0q_0e^{-\sigma_0(t-\tau^*)}+ \sigma_0q_0 e^{-\sigma_0(t-\tau^*)}\right)\\
  &\leq   B_1q_0(1+\sigma_0)e^{-\sigma_0(t-\tau^*)}
  \end{aligned}
\end{equation}
with  
\begin{equation*}
  B_1:=\frac{B_0\delta \|k^*\|_{C^1([0,T])}+d}{ \delta k_{min}}.
\end{equation*}

\begin{lemma}\label{lemma-upp}
There exists a positive constant $\gamma$ such that $\bar{u}$ and $\bar{h}$  given by \eqref{eq-ups} satisfy
\begin{equation*}
  \begin{cases}
    h(t)\leq \bar{h}(t) &\text{ for }t\geq \tau^*,\\
u(t,x)\leq \bar{u}(t,x) &\text{ for }t\geq\tau^*, ~x\in[g(t),h(t)].
  \end{cases}
\end{equation*}
\end{lemma}
\begin{proof}  We  show that $(\bar{u}(t,x),g(t),\bar{h}(t))$ is an upper solution of \eqref{free-bound} for $t\geq \tau^*$, provided that $\gamma$ is chosen properly.  

  Clearly,  $\bar{h}(\tau^*)>h(\tau^*)$. From the choice of $\tau^*$, the monotonicity of $\bar{\Phi}$ and \eqref{eq-1q_0} we also have 
\begin{align*}
  \bar{u}(\tau^*,x)&\geq  \bar{\Phi}(\tau^*, \bar{h}(\tau^*)-h(\tau^*)+z_\delta(\tau^*))+q_0\geq  \bar{\Phi}(\tau^*, K_0)+q_0\\
 & =\Phi^*(\tau^*, K_0)+q_0\geq 1+\frac{q_0}{2}\geq u(\tau^*,x) \text{ for }x\in[g(\tau^*),h(\tau^*)].
\end{align*}
Moreover,  by direct calculations,
\begin{eqnarray}\label{eq-sum}
  &&\bar{u}(t,\bar{h}(t))=\delta, ~~\bar{u}(t,g(t))\geq \bar{u}(t,\bar{h}(t))=\delta=u(t,g(t)),\nonumber\\
  &&\bar{u}(t,h(t))\geq   \bar{u}(t,\bar{h}(t))=\delta=u(t,h(t)) ~~~\text{ if }h(t)\leq \bar{h}(t),\\
  &&-\frac{d}{\delta}\bar{u}_x(t,\bar{h}(t))=\frac{d}{\delta} \bar{\Phi}_z(t,z_\delta(t))=\frac{d}{\delta} \Upsilon_z(t,z_\delta(t))\nonumber.
\end{eqnarray}
Here, $\bar{\Phi}_z$ and $\Upsilon_z$  denote the partial derivatives of $\bar{\Phi}(t,z) $ and $\Upsilon(t,z)$ with respect to $z$, respectively.
By \eqref{eq-zdt}, we also have 
 \begin{equation*}
   \frac{d}{\delta} \Upsilon_z(t,z_\delta(t))\leq \frac{d}{\delta} \Upsilon_z(t,0)+B_2|z_\delta(t)|\leq k^*(t)+B_2B_0q_0 e^{-\sigma_0(t-\tau^*)},
 \end{equation*}
 where 
 \begin{equation*}
  B_2:=\frac{d}{\delta}\|\Upsilon_{zz}(t,z)\|_{L^{\infty}([0,T]\times[-a,0])}.
 \end{equation*}
 Hence, 
 \begin{eqnarray*}
  -\frac{d}{\delta}\bar{u}_x(t,\bar{h}(t))=\frac{d}{\delta}\Upsilon_z(t,z_\delta(t))&\leq &  k^*(t)+B_2B_0q_0 e^{-\sigma_0(t-\tau^*)}\\
&\leq & k^*(t)+\gamma\sigma_0e^{-\sigma_0(t-\tau^{*})}= \bar{h}'(t)
 \end{eqnarray*}
 provided that 
 \begin{equation}\label{eq-gq1}
  \gamma\geq \frac{q_0 B_2B_0 }{\sigma_0}.
 \end{equation}

Denote $\zeta=\zeta (t,x):=\bar{h}(t)-x+z_\delta(t)$. We are going to show that 
\begin{eqnarray*}
  \mathcal{N}[\bar{u}]:=\bar{u}_t-d\bar{u}_{xx}-f(t,\bar{u})\geq 0 \mbox{ for } t>\tau^*,\  x\in[g(t),\bar{h}(t)],
\end{eqnarray*}
according to the cases  $\zeta(t,x)\leq 0$ and $\zeta(t,x)> 0$, respectively.
By the definition of $\bar{u}$, for  $-a\leq \zeta\leq 0$,
\begin{equation*}
  \bar{u}(t,x)=\Upsilon(t,\zeta)+ q_0e^{-\sigma_0(t-\tau^{*})}=-M\zeta^2+\frac{\delta}{d}k^*(t)\zeta+\delta+ q_0e^{-\sigma_0(t-\tau^{*})}.
\end{equation*}
So
\begin{eqnarray*}
  &&\bar{u}_t=\Upsilon_t+\Upsilon_\zeta(\bar{h}'(t)+z_\delta'(t))-\sigma_0q_0e^{-\sigma_0(t-\tau^*)},\\
  &&\bar{u}_x=- \Upsilon_\zeta=2M\zeta-\frac{\delta}{d}k^*(t),~~\bar{u}_{xx}= \Upsilon_{\zeta\zeta}=-2M,
\end{eqnarray*}
and by \eqref{eq-upes}, \eqref{eq-zd'} and the definition of $M$, we have, when $-a\leq\zeta\leq 0$,  
{\small \begin{eqnarray*}
  \mathcal{N}[\bar{u}]&=&\Upsilon_t+\Upsilon_\zeta(\bar{h}'(t)+z_\delta'(t))-\sigma_0q_0e^{-\sigma_0(t-\tau^*)}+2Md-f(t,\bar{u})\\
  &\geq & -\frac{\delta a}{d}\|k^{*}\|_{C^{1}([0,T])}+\frac{\delta  k_{min}}{d}\big[k^*(t)+\gamma\sigma_0e^{-\sigma_0(t-\tau^{*})}\big]-\left(2M a+\frac{\delta}{d}k_{max}\right) B_1q_0(1+\sigma_0)e^{-\sigma_0(t-\tau^*)}\\
  &&-\sigma_0q_0e^{-\sigma_0(t-\tau^*)}+2Md-f(t,\bar{u})\\
  &\geq &  -\frac{\delta a}{d}\|k^{*}\|_{C^{1}([0, T])}+\left[\frac{\delta  k_{min}}{d}\gamma\sigma_0-\left(2M a+\frac{\delta}{d}k_{max}\right) B_1q_0(1+\sigma_0)-\sigma_0q_0\right] e^{-\sigma_0(t-\tau^*)}\\
  && +2Md-f(t,\bar{u})\\
  &\geq& 0 \ \ \mbox{ provided that }
  \end{eqnarray*}}
\begin{equation}\label{eq-gq2}
  \gamma\geq \frac{dq_0}{\delta k_{min}\sigma_0}\left[\Big(2M a+\frac{\delta}{d} k_{max}\Big) B_1(1+\sigma_0)+\sigma_0\right].
\end{equation}

For $\zeta>0$, we have
\begin{eqnarray*}
&& \bar{u}(t,x)=\Phi^*(t, \zeta)+q_0e^{-\sigma_0(t-\tau^{*})},\\
 &&\bar{u}_t=\Phi_t^*+\Phi_\zeta^*(\bar{h}'(t)+z_\delta'(t))-\sigma_0q_0e^{-\sigma_0(t-\tau^{*})},\\
 &&\bar{u}_x=- \Phi_\zeta^*,  ~~ \bar{u}_{xx}=\Phi_{\zeta\zeta}^*,
\end{eqnarray*}
and  hence,
\begin{equation}\label{Nu^2}
\begin{aligned}
 \mathcal{N}[\bar{u}]&= \Phi_t^*+\Phi_\zeta^*[k^*(t)+\gamma\sigma_0e^{-\sigma_0(t-\tau^{*})}+z_\delta'(t)]-\sigma_0q_0e^{-\sigma_0(t-\tau^{*})}-d \Phi_{\zeta\zeta}^*-f(t,\bar{u})\\
&=[\gamma\sigma_0e^{-\sigma_0(t-\tau^{*})}+z_\delta'(t)]\Phi_\zeta^*-\sigma_0q_0e^{-\sigma_0(t-\tau^{*})}+f(t,\Phi^*)-f(t,\bar{u}).
\end{aligned}
\end{equation}
By the mean value theorem, we can express \(f(t, \bar{u})\) as  
\begin{equation*}
  f(t,\bar{u})=f(t,\Phi^*)+f_u(t,\Phi^*+\theta q_0e^{-\sigma_0(t-\tau^{*})})q_0e^{-\sigma_0(t-\tau^{*})}
\end{equation*}
for some $\theta=\theta(t,x)\in(0,1)$. 
Since  $\Phi^*_\zeta(t,\zeta)>0$ for $\zeta\geq 0$, there exist positive constants  $Q_0$   and $Q_1$ such that
 \begin{equation*}
  0<Q_0\leq \Phi^*_\zeta(t,\zeta)\leq Q_1 ~~\text{ for }t\in\mathbb{R},\ \zeta\in[0,K_0].
 \end{equation*}
 Hence we can use \eqref{eq-zd'} to obtain, for $\zeta\in[0,K_0]$,
\begin{eqnarray*}
  \mathcal{N}[\bar{u}]
  &\geq& \gamma\sigma_0Q_0e^{-\sigma_0(t-\tau^{*})}-B_1q_0Q_1(1+\sigma_0)e^{-\sigma_0(t-\tau^*)}-\sigma_0q_0e^{-\sigma_0(t-\tau^{*})}\\
  &&-f_u(t,\Phi^*+\theta q_0e^{-\sigma_0(t-\tau^{*})})q_0e^{-\sigma_0(t-\tau^{*})}\\
&\geq& \left[\gamma\sigma_0Q_0-B_1q_0Q_1(1+\sigma_0)-\sigma_0q_0-q_0\max_{t\in[0,T],u\in[\delta,1+q_0]}f_u(t,u)\right]e^{-\sigma_0(t-\tau^{*})}\\
&\geq& 0\ \ \mbox{ provided that }
\end{eqnarray*}
\begin{equation}\label{eq-gq3}
  \gamma\geq \frac{q_0\left[B_1Q_1(1+\sigma_0)+\sigma_0+\max_{t\in[0,T],u\in[\delta,1+q_0]}\partial_uf(t,u)\right]}{\sigma_0Q_0}.
\end{equation}

It remains to prove  $\mathcal{N}[\bar{u}]\geq 0$  for   $\zeta\geq K_0$.
Thanks to \eqref{eq-1q_0} and the choice of $q_0$, it is clear that  
\begin{equation*}
 1-\eta\leq  \Phi^*(t,\zeta)+\theta q_0e^{-\sigma_0(t-\tau^{*})}\leq 1+q_0\leq 1+\eta ~~\text{ for }\zeta\geq K_0.
\end{equation*}
Thus, for $\zeta\geq K_0$, making use of  \eqref{eq-f_usig}, we deduce from  \eqref{Nu^2} that 
\begin{eqnarray*}
  \mathcal{N}[\bar{u}]&\geq  &(\gamma\sigma_0e^{-\sigma_0(t-\tau^{*})}-B_1q_0(1+\sigma_0)e^{-\sigma_0(t-\tau^*)})\Phi_\zeta^*-\sigma_0q_0e^{-\sigma_0(t-\tau^{*})}\nonumber\\
&&-f_u(t,\Phi^*+\theta q_0e^{-\sigma_0(t-\tau^{*})})q_0e^{-\sigma_0(t-\tau^{*})}\\
&\geq &(\gamma\sigma_0-B_1q_0(1+\sigma_0))\Phi_\zeta^* e^{-\sigma_0(t-\tau^{*})} \nonumber\\
&&+\left(\min_{t\in[0,T],u\in[1-\eta,1+\eta]}[-f_u(t,u)]-\sigma_0\right)q_0e^{-\sigma_0(t-\tau^{*})}\\
&\geq& 0\ \ \mbox{ provided that }
\end{eqnarray*}
\begin{equation}\label{eq-gq4}
  \gamma\geq \frac{B_1q_0(1+\sigma_0)}{\sigma_0}.
\end{equation}
Therefore, upon choosing $\gamma$ sufficiently large so that   \eqref{eq-gq1}, \eqref{eq-gq2}, \eqref{eq-gq3} and \eqref{eq-gq4} are fulfilled,  we may apply Lemma \ref{lem-comp} to conclude that 
\begin{equation*}
  h(t)\leq \bar{h}(t),\ u(t,x)\leq \bar{u}(t,x) ~~\text{ for }t\geq \tau^*,~x\in[g(t),h(t)].
\end{equation*}
This completes the proof.
\end{proof}

 \subsection{Lower bound}
 Fix $q_1\in(0,1-\delta)$ and define
 \begin{equation}\label{eq-low}
  \begin{cases}
  \displaystyle\underline{h}(t):=\int_0^{t}k^*(s)ds+\beta e^{-\sigma_1t}+K_1,  &t\geq0,\\
  \underline{u}(t,x):=\Phi^*(t,\underline{h}(t)-x+x_\delta(t))+\Phi^*(t,\underline{h}(t)+x+x_\delta(t))-1-q_1e^{-\sigma_1t}, &x\in[-\underline{h}(t),\underline{h}(t)],
 \end{cases}
 \end{equation}
where  $\beta$, $\sigma_1$, $K_1$ are  positive constants to be determined later and $x_\delta(t)>0$ is determined by  
\begin{equation}\label{delta}
\Phi^*(t,x_\delta(t))+\Phi^*(t,2\underline{h}(t)+x_\delta(t))-1-q_1e^{-\sigma_1t}=\delta.
 \end{equation}
By the choice of $q_1$ and the fact that $\Phi^*(\cdot,\infty)=1$, we  know that   $x_\delta(t)$ is well-defined, bounded and  satisfies $x_\delta(\infty)=0$. Thus, there exists a positive constant $N_2$ such that 
\begin{equation}\label{phiz-bound}
  \max\limits_{t\in[0,\infty),x\in[0,x_\delta(t)]}\Big(|\Phi_{xx}^*(t,x)|+|\Phi^*_{tx}(t,x)|+|\Phi^*_{t}(t,x)|\Big)\leq N_2.
\end{equation}
Moreover, since $\Phi^*_x>0$, there exist positive constants $N_3$ and $N_4$  such that 
\begin{equation}\label{phiz-bound+}
  0<N_3\leq \Phi_{x}^*(t,x)\leq N_4 ~\text{ for }t\in[0,\infty),\ x\in[0,x_\delta(t)].
\end{equation}
By Lemma \ref{lem-1-phi}, for any fixed $\mu\in(0,\mu_0)$,  that there exists $N_1$ such that    \eqref{phi-infty} holds.
Since
\[2\underline{h}(t)+x_\delta(t)\geq  k_{min}\,t ~\text{ for }t\geq 0\]
with $k_{min}$ given by \eqref{eq-k+-}, we obtain from \eqref{phi-infty} and  \eqref{delta} that 
\begin{eqnarray*}
  \Phi^*(t,x_\delta(t))&=&\delta+1+q_1e^{-\sigma_1t}-\Phi^*(t,2\underline{h}(t)+x_\delta(t))\\
  &\leq& \delta+q_1e^{-\sigma_1t}+N_1 e^{-\mu  k_{min}\,t}\ \mbox{ for }t\geq0.
\end{eqnarray*}
The  mean value theorem then yields
\begin{equation*}
  \Phi_x^*(t,\theta(t))x_\delta(t)=\Phi^*(t,x_\delta(t))-\Phi^*(t,0)\leq  q_1e^{-\sigma_1t}+N_1 e^{-\mu  k_{min}\,t}
\end{equation*}
for some  $\theta(t)\in[0,x_\delta(t)]$. Combining this  with  \eqref{phiz-bound+} we deduce 
\begin{equation}\label{eq-xdt}
  x_\delta(t)\leq C_1q_1e^{-\sigma_1t}+C_1 e^{-\mu  k_{min}\,t} ~\text{ for } t\geq 0, 
\end{equation}
where 
\begin{equation*}
  C_1:=\max\left\{\frac{1 }{N_3},\frac{N_1}{N_3}\right\}.
\end{equation*}
Due to $\Phi^*(t,0)\equiv\delta$ we have  $\Phi_t^*(t,0)=0$, and thus, by the  mean value theorem and  \eqref{phiz-bound},  we deduce 
\begin{eqnarray*}
  |\Phi_t^*(t,x_\delta(t))|=  |\Phi_t^*(t,x_\delta(t))-\Phi_t^*(t,0)|\leq |\Phi^*_{tx}(t,\theta(t))|x_\delta(t)\leq N_2x_\delta(t),
\end{eqnarray*}
where  $\theta(t)\in[0,x_\delta(t)]$. Differentiating equation \eqref{delta}  and employing inequalities \eqref{phi-infty} and \eqref{phiz-bound}, we obtain
\begin{eqnarray*}
  ~~&&[\Phi^*_x(t,x_\delta(t))+\Phi^*_x(t,2\underline{h}(t)+x_\delta(t))]x_\delta'(t)\\
  &=&-\Phi^*_t(t,x_\delta(t))-\Phi^*_t(t,2\underline{h}(t)+x_\delta(t))-2(k^*(t)-\sigma_1\beta e^{-\sigma_1t})\Phi^*_x(t,2\underline{h}(t)+x_\delta(t))-\sigma_1q_1e^{-\sigma_1t}\\
  &\leq&N_2x_\delta(t)+N_1(1+2\sigma_1\beta e^{-\sigma_1t})e^{-\mu  k_{min}\,t}-\sigma_1q_1e^{-\sigma_1t}\\
  &\leq & (N_2C_1-\sigma_1)q_1e^{-\sigma_1t}+(N_2C_1+N_1+2N_1\sigma_1\beta e^{-\sigma_1t})e^{-\mu  k_{min}\,t}\\
  &\leq &\tilde{C_1}q_1e^{-\sigma_1t}+\tilde{C_1}(1+\sigma_1\beta e^{-\sigma_1t})  e^{-\mu  k_{min}t},
\end{eqnarray*}
where 
\begin{equation*}
  \tilde{C_1}:=N_2C_1+2N_1-\sigma_1.
\end{equation*}
This and \eqref{phiz-bound+} imply 
\begin{equation}\label{eq-xdelta'}
  x_\delta'(t)\leq C_2q_1e^{-\sigma_1t}+C_2(1+\sigma_1\beta  e^{-\sigma_1t})  e^{-\mu  k_{min}t} ~\text{ for }t\geq 0,
\end{equation}
with  $C_2:=\frac{\tilde{C_1}}{N_3}>0$.
\begin{lemma}\label{lemma-lower}
  There exist positive constants $\beta$, $\sigma_1$, $K_1$ in $\mathbb R$ and  $n_0, n_2$ in $\mathbb{N}$ such that   $\underline{u}$ and $\underline{h}$  given by \eqref{eq-low} satisfy
  \begin{equation*}
    \begin{cases}
[-\underline{h}(t),\underline{h}(t)]\subset (g(t+n_2T),h(t+n_2T)) &\text{ for }t\geq  n_0T,\\
   \underline{u}(t,x)\leq u(t+n_2T,x) &\text{ for }t\geq n_0T, ~x\in[-\underline{h}(t), \underline{h}(t)].
    \end{cases}
  \end{equation*}
  \end{lemma}

  \begin{proof} It suffices to show that $(\underline{u},-\underline{h},\underline{h})$ is  a lower solution of \eqref{free-bound}. For technical convenience, we extend $f$ to   $(-2,0]$ so  that $f\in C^1$ in $u$ for $u\geq -2$, uniformly with respect to $t\in \mathbb R$. Note that this extension does not affect the asymptotic behavior of $u$. By \eqref{phiz-bound} and \eqref{eq-xdt}, we obtain 
\begin{eqnarray*}
  \frac{d}{\delta}\Phi^*_x(t,x_\delta(t))=   \frac{d}{\delta}\Phi^*_x(t,0)+  \frac{d}{\delta}\Phi^*_{xx}(t,\theta(t))x_\delta(t)\geq k^*(t)-\frac{d}{\delta}N_2C_1 (q_1e^{-\sigma_1t}+ e^{-\mu  k_{min}\,t}).
\end{eqnarray*}
Hence,
\begin{eqnarray*}
  -\frac{d}{\delta}\underline{u}_x(t,\underline{h}(t))&=&-\frac{d}{\delta}\left[\Phi^*_x(t,2\underline{h}(t)+x_\delta(t))-\Phi^*_x(t,x_\delta(t))\right]\\
  &\geq& k^*(t)-\frac{d}{\delta}N_2C_1 (q_1e^{-\sigma_1t}+ e^{-\mu  k_{min}t})-\frac{d}{\delta}N_1 e^{-\mu  k_{min}t}\\
  &\geq &k^*(t)- \frac{d}{\delta}N_2C_1q_1 e^{-\sigma_1t}-\frac{d}{\delta}(N_2C_1+N_1) e^{-\mu  k_{min}t}.
\end{eqnarray*}
It follows that 
\begin{eqnarray*}
  \underline{h}'(t)=k^*(t)-\sigma_1\beta e^{-\sigma_1t}\leq   -\frac{d}{\delta}\underline{u}_x(t,\underline{h}(t)) ~\text{ for }t\geq 0
\end{eqnarray*}
provided $\sigma_1<\mu  k_{min}$ and 
\begin{equation*}
  \beta\geq \frac{(2N_2C_1+N_1)d}{\sigma_1\delta}.
\end{equation*}

 Since  $\Phi^*(t,\infty)=1$ and $\Phi^*_x(t,\infty)=0$ uniformly in $t\in\mathbb{R}$, for any given  small positive $\epsilon_0$,  there exists  $\tilde{K}_0=\tilde{K}_0(\epsilon_0)>0$ large enough such that 
\begin{equation}\label{phi-epsi}
  \Phi^*(t,x)> 1-\epsilon_0 \text{ and }\Phi_x(t,x)<\epsilon_0/2~\text{ for }t\in\mathbb{R},\ x\geq  \tilde{K}_0.
\end{equation}
Therefore, if $K_1$ in \eqref{eq-low} is chosen to satify $K_1\geq \tilde{K}_0$, then   $\underline{h}(t)\geq \tilde{K}_0$ for $t\geq 0$.

Next, we shall prove
\begin{equation*}
  \mathcal{N}[\underline{u}]=\underline{u}_t-d\underline{u}_{xx}-f(t,\underline{u})\leq 0.
\end{equation*}
For convenience, denote 
\[
\xi^-=\xi^-(t,x):=\underline{h}(t)-x+x_\delta(t),\ \ \xi^+=\xi^+(t,x):=\underline{h}(t)+x+x_\delta(t).
\]
 By direct calculation we have 
\begin{eqnarray*}
  \underline{u}_t&=& \Phi_t^*(t,\xi^-)+\Phi^*_t(t,\xi^+)+(\Phi^*_z(t,\xi^-)+\Phi^*_z(t,\xi^+))(\underline{h}'(t)+x_\delta'(t))+\sigma_1q_1 e^{-\sigma_1t},\\
  \underline{u}_{xx}&=& \Phi^*_{zz}(t,\xi^-)+\Phi^*_{zz}(t,\xi^+),
\end{eqnarray*}
where $\Phi^*_z$ denotes the partial derivative of $\Phi^*(t,z) $ with respect to $z$.
Since $\Phi^*$ satisfies \eqref{eqn-wave}, we obtain
\begin{equation}\label{eq-Nu_}
\begin{aligned}
  \mathcal{N}[\underline{u}]&= (\Phi^*_z(t,\xi^-)+\Phi^*_z(t,\xi^+))(-\sigma_1\beta e^{-\sigma_1t}+x_\delta'(t))+\sigma_1q_1 e^{-\sigma_1t}\\
  &\ \ \ \ +f(t,\Phi^*(t,\xi^-))+f(t,\Phi^*(t,\xi^+))-f\left(t,\bar{u}\right).
  \end{aligned}
\end{equation}
We next show  $\mathcal{N}[\underline{u}]\leq 0$ in the following regions one by one.
\begin{itemize}
\item[(a):] $ \underline{h}(t)-\tilde{K}_0\leq x\leq \underline{h}(t)$, 
\item[(b):] $  -\underline{h}(t)\leq x\leq -\underline{h}(t)+\tilde{K}_0$, 
\item[(c):] $  -\underline{h}(t)+\tilde{K}_0\leq x\leq \underline{h}(t)-\tilde{K}_0$.
\end{itemize}
\medskip

(a) For $\underline{h}(t)-\tilde{K}_0\leq x\leq \underline{h}(t)$, we have $x_\delta(t)\leq \xi^-\leq x_\delta(t)+\tilde{K}_0$  and
\begin{equation*}
  \xi^+=\underline{h}(t)+x+x_\delta(t)\geq 2\underline{h}(t)-\tilde{K}_0+x_\delta(t)\geq \underline{h}(t)\geq k_{min}\, t
\end{equation*}
with $x_\delta(t)$ being bounded and $t\geq n_0T\gg 1$.
Thus in this case there exist positive constants $\tilde{N}_3$ and $\tilde{N}_4$ such that 
\begin{equation*}
 0< \tilde{N}_3\leq\Phi^*_z(t,\xi^+)+\Phi^*_z(t,\xi^-)\leq \tilde{N}_4.
\end{equation*}
Since $f\in C^1$ in $u$ uniformly in $t$, and $f(\cdot,1)=0$, 
we can find $L>0$ such that
\begin{equation*}
  f(t,\Phi^*(t,\xi^+))= f(t,\Phi^*(t,\xi^+))-f(t,1)\leq L|\Phi^*(t,\xi^+)-1|\leq LN_1e^{-\mu\xi^+} \leq LN_1 e^{-\mu  k_{min}t}
\end{equation*}
and
\begin{eqnarray*}
  f(t,\Phi^*(t,\xi^-))-f(t,\underline{u})\leq L|1-\Phi^*(t,\xi^+)+q_1e^{-\sigma_1t}|\leq L(N_1 e^{-\mu  k_{min}t}+q_1e^{-\sigma_1t}).
\end{eqnarray*}
 It then follows from \eqref{eq-Nu_} and \eqref{eq-xdelta'} that 
\begin{eqnarray*}
  \mathcal{N}[\underline{u}]&\leq&(\tilde{N}_4C_2q_1-\tilde{N}_3\sigma_1\beta) e^{-\sigma_1t}+\tilde{N}_4C_2 (1+\sigma_1\beta  e^{-\sigma_1t})e^{-\mu  k_{min}t}+\sigma_1q_1 e^{-\sigma_1t}\\
  && +L(2N_1 e^{-\mu  k_{min}t}+q_1e^{-\sigma_1t})\\
  &=&\left(\tilde{N}_4C_2q_1-\tilde{N}_3\sigma_1\beta+\sigma_1q_1  +Lq_1\right) e^{-\sigma_1t}+[\tilde{N}_4C_2(1+\sigma_1\beta  e^{-\sigma_1t} )+2LN_1] e^{-\mu  k_{min}t}\\
  &=& \left(\tilde{N}_4C_2q_1-\tilde{N}_3\sigma_1\beta/2+\sigma_1q_1  +Lq_1\right) e^{-\sigma_1t}+(\tilde{N}_4C_2+2LN_1) e^{-\mu  k_{min}t}\\
  &&+\sigma_1\beta(\tilde{N}_4C_2 e^{-\mu  k_{min}t}-\tilde{N}_3 /2)e^{-\sigma_1t}\leq 0 ~~\text{ for }t\geq  n_0T\gg1,
\end{eqnarray*}
provided that $\sigma_1<\mu k_{min}$ and
\begin{equation*}
  \beta\geq \frac{2(\tilde{N}_4C_2q_1+\sigma_1q_1+Lq_1+\tilde{N}_4C_2+2LN_1)}{\tilde{N}_3\sigma_1}.
\end{equation*} 
Here we also used $ n_0\gg 1$ and so 
\begin{equation*}
  \tilde{N}_3\geq 2\tilde{N}_4C_2 e^{-\mu k_{min} n_0T}.
\end{equation*}

(b) Similarly, we can show  $\mathcal{N}[\underline{u}]\leq 0$ for $-\underline{h}(t)\leq x\leq -\underline{h}(t)-\tilde{K}_0$.
 \medskip
 
(c) In  the remaining case   $ -\underline{h}(t)+\tilde{K}_0\leq x\leq \underline{h}(t)-\tilde{K}_0$, it is easily seen that  $\xi^+,\xi^-\geq \tilde{K}_0$ and $\max\{\xi^+,\xi^-\}\geq \underline{h}(t)\geq  k_{min}t$. This together with \eqref{phi-epsi} implies that, for $t\geq n_0T\gg1$,  
\begin{equation}\label{eq-mami}
   \begin{cases} \min\{1-\Phi^*(t,\xi^+), 1-\Phi^*(t,\xi^-)\}\leq N_1e^{-\mu k_{min}t},\ \\ 
   \max\{1-\Phi^*(t,\xi^+), 1-\Phi^*(t,\xi^-)\}\leq \epsilon_0,\ \\
   \Phi^*_z(t,\xi^+)+\Phi^*_z(t,\xi^-)<\epsilon_0.
   \end{cases}
\end{equation}

For any $t\geq 0$ and $u,v\in[0,1]$, we denote
\begin{equation*}
  F(u,v)(t):=f(t,u)+f(t,v)-f\left(t, u+v-1-q_1e^{-\sigma_1 t}\right).
\end{equation*}
Since $u, v\in [0,1]$ and $f(t,\cdot)$ is $C^1$  uniformly in $t$ with $f(t,1)\equiv 1$, there exists some $L>0$ such that
\begin{eqnarray*}
  F(u,v)(t)&\leq& \min\left\{L(1-u)+f_u(t,\theta_v)(1+ q_1e^{-\sigma_1t}-u), L(1-v)+f_u(t,\theta_u)(1+ q_1e^{-\sigma_1t}-v)\right\}\\
  &\leq& \min\left\{2L(1-u)+f_u(t,\theta_v) q_1e^{-\sigma_1t}, 2L(1-v)+f_u(t,\theta_u) q_1e^{-\sigma_1t}\right\},
\end{eqnarray*}
where $\theta_v\in [u+v-1-q_1e^{-\sigma_1 t}, v]$, $\theta_u\in [u+v-1-q_1e^{-\sigma_1 t}, u]$.
We may now use \eqref{eq-mami} to deduce that
\begin{equation*}
  F(\Phi^*(t,\xi^+),\Phi^*(t,\xi^-))\leq 2LN_1 e^{-\mu  k_{min}t}+\sigma_{\epsilon_0}q_1e^{-\sigma_1t} \mbox{ for } t\geq n_0T,
\end{equation*}
where
\[
\sigma_{\epsilon_0}:= \max_{t\in [0,T], u\in [1-3\epsilon_0, 1]}f_u(t,u)\leq \frac{1}{2}\max_{t\in [0,T]}f_u(t,1)<0 
\]
 provided that $\epsilon_0$ is small enough and $n_0\gg 1$.
The above inequality,  \eqref{eq-xdelta'}, \eqref{eq-Nu_}  and \eqref{eq-mami} then lead to,  for $t\geq  n_0T\gg1$,
 \begin{eqnarray*}
\mathcal{N}[\underline{u}]&\leq & (\Phi^*_z(t,\xi^+)+\Phi^*_z(t,\xi^-))x_\delta'(t)+\sigma_1q_1 e^{-\sigma_1t}+F(\Phi^*(t,\xi^+),\Phi^*(t,\xi^-))\\
&\leq &C_2q_1\epsilon_0e^{-\sigma_1t}+C_2\epsilon_0(1+\sigma_1\beta  e^{-\sigma_1t})e^{-\mu  k_{min}t}+(\sigma_1+\sigma_{\epsilon_0})q_1 e^{-\sigma_1t}+2LN_1 e^{-\mu  k_{min}t}\\
&\leq&q_1e^{-\sigma_1t}\left\{C_2\epsilon_0+\sigma_1+\frac{1}{2}\max_{t\in [0,T]}f_u(t,1)
+e^{-(\mu  k_{min}-\sigma_1)t}\left[C_2\epsilon_0(1+\sigma_1\beta e^{-\sigma_1t})+2LN_1\right]\right\}\\
&\leq& 0 \ \  \mbox{ provided that $\epsilon_0$ is small enough and}
\end{eqnarray*}
  \[
  \displaystyle\sigma_1<\min\Big\{-\frac12 \max_{t\in[0,T]}f_u(t,1),\mu k_{min}\Big\}.
  \]

 We have now proved that $\mathcal{N}[\underline{u}]\leq 0$ for $t\geq n_0T\gg1$ and $x\in [-\underline h(t), \underline h(t)]$, with suitable choices of the parameters.
\smallskip

By Proposition \ref{prop-summ}, we can find positive integers $n_1\geq  n_0$ such that 
\begin{equation*}\begin{cases}
-g(n_1T), h(n_1T)>\underline{h}( n_0T),\\
  u(t,x)\geq \delta ~\text{ for }t\geq n_1T,\  x\in[g(t),h(t)],\\
  u(n_1T,x)\geq 1-q_1e^{-\sigma_1 n_0T}\geq \underline{u}( n_0T,x) ~\text{ for }x\in[-\underline{h}( n_0T),\underline{h}( n_0T)].
  \end{cases}
\end{equation*} 
Denote $n_2=n_1-n_0$. Then
\begin{equation*}
u( n_0T+n_2T,x)\geq \underline{u}( n_0T,x) ~\text{ for }x\in[-\underline{h}( n_0T),\underline{h}( n_0T)],
\end{equation*} 
and for $t\geq  n_0T$, 
 \begin{equation*}
  u(t+n_2T,\pm \underline{h}(t))\geq \delta=\underline{u}(t,\pm \underline{h}(t)) \text{ whenever } [-\underline{h}(t),\underline{h}(t)]\subset [g(t+n_2T),h(t+n_2T)].
 \end{equation*}  
So we are now able to apply the lower solution version of Lemma  \ref{lem-comp2} to conclude that 
\begin{eqnarray*}\begin{cases}
  [-\underline{h}(t),\underline{h}(t)]\subset (g(t+n_2T),h(t+n_2T)),\\ \underline{u}(t,x)\leq u(t+n_2T,x),
  \end{cases}  \text{ for }t\geq  n_0T,\ x\in [-\underline{h}(t),\underline{h}(t)].
\end{eqnarray*}
The proof is thereby complete. 
\end{proof}

\begin{proposition}\label{prop-bound}
  There exists $Q>0$ such that 
  \begin{equation}
     \Big|g(t)+\int_0^tk^*(s)ds\Big|+ \Big|h(t)-\int_0^tk^*(s)ds\Big|\leq Q ~~\text{ for } t\geq 0.
  \end{equation}
\end{proposition}
\begin{proof}
  Denote $\tau_0:=\max\{\tau^{*}, (n_2+n_0)T\}$. By Lemmas \ref{lemma-upp} and  \ref{lemma-lower}, we have for $t\geq \tau_0$,
  \begin{eqnarray*}
    -\int_{t-n_2T}^tk^*(s)ds+\beta e^{-\sigma_1(t-n_2T)}+K_1&\leq&  h(t)-\int_0^tk^*(s)ds\\
    &\leq& -\int_0^{\tau^*}k^*(s)ds+\gamma(1-e^{-\sigma_0(t-\tau^{*})})+h(\tau^*)+a+K_0.
  \end{eqnarray*}
Set 
\begin{equation*}
  Q:=2\max\left\{n_2\overline{k^*}T+\beta +K_1,\int_0^{\tau^*}k^*(s)ds+\gamma+h(\tau^*)+a+K_0, \max_{t\in[0,\tau_0]}\Big|h(t)-\int_0^tk^*(s)ds\Big|\right\}.
\end{equation*}
Then 
\begin{equation*}
  \Big|h(t)-\int_0^tk^*(s)ds\Big|\leq Q/2 \text{ for all }t\geq 0.
\end{equation*}
Similarly, we can show
\begin{equation*}
  \Big|g(t)+\int_0^tk^*(s)ds\Big|\leq Q/2 \text{ for all }t\geq 0.
\end{equation*}
The proof is complete.
\end{proof}

\subsection{Convergence along a special time sequence}

Let $(u(t,x), g(t), h(t))$ be the solution to \eqref{free-bound}, $Q$ be given by Proposition \ref{prop-bound},  and
\begin{equation*}\begin{cases}
  l(t):=\int_0^tk^*(s)ds-2Q, \\
  \sigma(t):=g(t)-l(t),\\
   \rho(t):=h(t)-l(t),\\  v(t,x):=u(t,x+l(t)).
   \end{cases}
\end{equation*}
Clearly, 
\begin{equation}\label{eq-c3c}
\lim_{t\to\infty}\sigma(t)=-\infty,\    Q\leq \rho(t)\leq 3Q,
\end{equation}
and
\begin{equation*}
u_x=v_x, ~u_{xx}=v_{xx},~  u_t=v_t-k^*(t)v_x.
\end{equation*}
Hence, $(v(t,x),\sigma(t),\rho(t))$ solves
\begin{equation*}
  \begin{cases}
    v_t-dv_{xx}-k^*(t)v_x=f(t,v), &t>0,\ \sigma(t)<x<\rho(t),\\
   v(t,\rho(t))=\delta, &t>0,\\
    \rho'(t)=h'(t)-l'(t)=-\frac{d}{\delta}v_x(t,\rho(t))-k^*(t),&t>0.\\
  \end{cases}
\end{equation*}
Let $\{t_n\}$ be a positive sequence converging to $\infty$ such that   
\begin{equation*}
  \lim_{n\to\infty}\rho(t_n)=\liminf_{t\to\infty}\rho(t)\in [Q, 3Q].
\end{equation*}
Write $t_n=k_nT+r_n$ with $k_n$ a positive integer and $r_n\in [0, T)$. Up to extraction of a subsequence  we may assume $r_n\to r_*\in[0,T]$.
Define
\begin{equation*}\begin{cases}
  (k_n^*(t),\sigma_n(t),\rho_n(t)):=(k^*(t+t_n-r_*),\sigma(t+t_n-r_*),\rho(t+t_n-r_*)), \\
  v_n(t,x):=v(t+t_n-r_*,x).
  \end{cases}
\end{equation*}
We have  the following result.
\begin{lemma}\label{lem-lim}
  There exists a subsequence of $\{(\rho_n, v_n)\}$, still denoted by  itself, such that
  \begin{equation*}
    \rho_n\to \Gamma \text{ in } C^{1+\alpha/2}_{loc}(\mathbb{R}) \text{  and }\|v_n-V\|_{C_{loc}^{(1+\alpha)/2,1+\alpha}(\overline\Sigma)}\to 0 ~\text{ as } n\to \infty,
  \end{equation*}
  where $\Sigma:=\{(t,x):t\in\mathbb{R}, -\infty<x< \Gamma(t)\}$. Moreover $(V(t,x),\Gamma(t))$ satisfies
  \begin{equation}\label{eq-V}
    \begin{cases}
      V_t-dV_{xx}-k^*(t)V_x=f(t,V), ~V>\delta, &(t,x)\in\Sigma,\\
      V(t,\Gamma(t))=\delta, &t\in\mathbb{R},\\
      \Gamma'(t)=-\frac{d}{\delta}V_x(t,\Gamma(t))-k^*(t),~ \Gamma(t)\geq \Gamma(r_*), &t\in\mathbb R.
    \end{cases}
  \end{equation}
\end{lemma}
\begin{proof}
  By Proposition \ref{prop-summ}, there exists $M_0>0$ such  that $|h'(t)|\leq M_0$, from which we have
  \begin{equation*}
    -M_0-k^*(t)\leq \rho'(t)\leq M_0-k^*(t) ~~\text{ for }t\geq 0.
  \end{equation*}
Define
\begin{equation*}
  \xi:=\frac{x}{\rho_n(t)}, ~~\tilde{V}_n(t,\xi):=v_n(t,x).
\end{equation*}
Then $(\tilde{V}_n(t,\xi),\rho_n(t))$ satisfies
\begin{equation}\label{eq-tilV}
  \begin{cases}
   \partial_t\tilde{V}_n-\frac{d}{\rho_n^2(t)}\partial_{\xi\xi}\tilde{V}_n-\frac{k^*_n(t)+\rho_n'(t)\xi}{\rho_n(t)}\partial_\xi\tilde{V}_n=f(t+t_n-r_*,\tilde{V}_n),&t>-t_n+r_*,\frac{\sigma_n(t)}{\rho_n(t)}\leq\xi\leq 1,\\[2mm]
   \tilde{V}_n(t,1)=\delta,~ \rho_n'(t)=-\frac{d}{\delta\rho_n(t)}\partial_\xi\tilde{V}_n(t,1)-k^*_n(t), &t>-t_n+r_*.
\end{cases}
\end{equation}
Since $u(t,x)$ is uniformly bounded for $t\in[0,\infty)$ and $x\in[g(t),h(t)]$, we know that $\tilde{V}_n$ is uniformly bounded in $\{(t,\xi):t\geq -t_n+r_*, \frac{\sigma_n(t)}{\rho_n(t)}\leq\xi\leq 1 \}$.  Thanks to \eqref{eq-c3c}, for any $\tau>0$, $L>0$  and $s\in\mathbb{R}$, we can apply the standard  interior and boundary $L^p$ estimates to \eqref{eq-tilV} over $[s-\tau,s]\times[-L,1]$ to obtain that, for any $p>1$,
\begin{equation*}
  \|\tilde{V}_n\|_{W_p^{1,2}([s-\tau,s]\times[-L,1])}\leq C_p \text{ for all large }n,
\end{equation*}
where $C_p$ is a constant depending only on $p$, $\tau$ and $L$ but independent of $n$ and $s$. It then follows from the  Sobolev  embedding  theorem that 
\begin{equation*}
  \|\tilde{V}_n\|_{C^{(1+\beta)/2,1+\beta}([s-\tau,s]\times[-L,1])}\leq \tilde{C}_p  \text{ for all large }n,
\end{equation*}
where $\alpha<\beta<1$ and $\tilde{C}_p$ is a constant depending  on $\beta$, $\tau$ and $L$ but independent of $n$ and $s$.
This together with \eqref{eq-tilV} and \eqref{eq-c3c} implies that, for any $s\in\mathbb{R}$, there exists $C_1>0$ depending  on $\beta$, $\tau$ and $L$ but independent of $n$ and $s$ such that 
\begin{equation*}
 Q\leq \rho_n\leq 3Q,\  \|\rho_n\|_{C^{1+\beta/2}([s-\tau,s])}\leq C_1.
\end{equation*}
Hence, up to extraction of a subsequence  through a diagonal process, we have 
\begin{equation*}
  \tilde{V}_n\to \bar{V} \text{ in } C^{(1+\alpha)/2, 1+\alpha}_{loc}(\mathbb R\times (-\infty,1]), ~~\rho_n\to \Gamma\text{ in } C_{loc}^{1+\alpha/2}(\mathbb R).
\end{equation*}
By applying  the standard parabolic regularity theory, one obtains that $(\bar{V}, \Gamma)$ solves the following  problem in the classical sense
\begin{equation*}
  \begin{cases}
  \bar{V}_t-\frac{d}{\Gamma^2(t)}\bar{V}_{\xi\xi}-\frac{k^*(t)+\Gamma'(t)\xi}{\Gamma(t)}\bar{V}_\xi=f(t,\bar{V}), &t\in\mathbb R,~\xi\in(-\infty,1],\\
  \bar{V}(t,1)=\delta,~ \Gamma'(t)=-\frac{d}{\delta\Gamma(t)}\bar{V}_\xi(t,1)-k^*(t), &t\in\mathbb R.
   \end{cases}
   \end{equation*}
It follows from Proposition \ref{prop-summ} that $\bar{V}\geq \delta$ in $\mathbb{R}\times (-\infty,1]$. Due to $f(t,\delta)>0$, we can apply the strong parabolic maximum principle to conclude that $\bar{V}>\delta$ in $\mathbb{R}\times (-\infty,1)$.  Set  $V(t,x)=\bar{V}(t,x/\Gamma(t))$. It is easily seen that $(V,\Gamma)$ solves \eqref{eq-V} and  satisfies
\begin{equation*}
  \lim_{n\to\infty}\|v_n-V\|_{C_{loc}^{(1+\alpha)/2,1+\alpha}(\overline\Sigma)}= 0.
\end{equation*}
Moreover, since 
\begin{equation*}
  \Gamma(r_*)=\lim_{n\to\infty} \rho(t_n)=\liminf_{t\to\infty}\rho(t) \text{ and } \Gamma(t)=\lim_{n\to\infty}\rho(t+t_n-r),
\end{equation*}
 we have  $\Gamma(t)\geq \Gamma(r_*)$ for $t\in\mathbb R$. The proof is thus complete.
\end{proof}

\subsection{Identification of  the limit pair $(V,\Gamma)$.}

By \eqref{eq-c3c}, we have
\[Q\leq \Gamma(t)\leq 3 Q, \quad\forall t\in\mathbb R.\]
From Lemma \ref{lemma-lower} and the $T$-periodic property of $\Phi^*$ in $t$, we see that for $t\geq (n_0+n_2)T$ and $x\in[-\underline{h}(t-n_2T),\underline{h}(t-n_2T)]$,
  \begin{eqnarray*}
  u(t,x)&\geq & \Phi^*(t, -x+\underline{h}(t-n_2T)+x_\delta(t-n_2T))\\
  &&+\, \Phi^*(t, x+\underline{h}(t-n_2T)+x_\delta(t-n_2T)) -1-q_1e^{-\sigma_1(t-n_2T)}.
\end{eqnarray*}
Denote 
{\small\begin{eqnarray*}\begin{cases}
  \Phi_n^-(t,x):=\Phi^*(t+r_n-r_*,-x-l(t+t_n-r_*)+\underline{h}(t+t_n-r_*-n_2T)+x_\delta(t+t_n-r_*-n_2T)),\\
  \Phi_n^+(t,x):=\Phi^*(t+r_n-r_*,\ \ x+l(t+t_n-r_*)+\underline{h}(t+t_n-r_*-n_2T)+x_\delta(t+t_n-r_*-n_2T)).
  \end{cases}
\end{eqnarray*}}
It follows that   
 \begin{equation}\label{eq-vn} 
  v_n(t,x)\geq \Phi_n^-(t,x)+\Phi_n^+(t,x)-1- q_1e^{-\sigma_1(t+t_n-r_*-n_2T)},
 \end{equation}
 for $t+t_n-r_*\geq (n_0+n_2)T$ and
\[
 x\in[-\underline{h}(t+t_n-r_*-n_2T)-l(t+t_n-r_*),\underline{h}(t+t_n-r_*-n_2T)-l(t+t_n-r_*)].
\]
 Since $x_\delta(t)$ is bounded and 
 \begin{equation*}
  r_n\to r_*  ~\text{ and }  \underline{h}(t+t_n-r_*-n_2T)+l(t+t_n-r_*)\to \infty ~~\text{ as }n\to\infty,
 \end{equation*}
 we obtain 
 \begin{equation*}
  \lim_{n\to\infty}\Phi_n^+(t,x)=1 \text{ locally uniformly in } \mathbb{R}\times\mathbb{R}.
 \end{equation*}
Moreover, direct calculations give
\begin{eqnarray*}
  &&\underline{h}(t+t_n-r_*-n_2T)-l(t+t_n-r_*)\\
  &\geq& \int_{0}^{t+t_n-r_*-n_2T}k^*(s)ds+K_1-\int_{0}^{t+t_n-r_*}k^*(s)ds+2Q\\
&=&-\int_{t+t_n-r_*-n_2T}^{t+t_n-r_*}k^*(s)ds+K_1+2Q\\
&\geq &-n_2T\overline{k^*} +K_1+2Q~ \text{ for }t+t_n-r_*\geq (n_0+n_1)T.
\end{eqnarray*}
Hence, there exists $L_0\in \mathbb R$  such that 
\begin{equation*}
  \underline{h}(t+t_n-r-n_2T)-l(t+t_n-r)\geq L_0 ~~\text{ for  }t+t_n-r\geq (n_0+n_1)T.
\end{equation*}
Since $x_\delta(t)\to 0$ as $t\to\infty$, by letting $n\to\infty$ in \eqref{eq-vn} and utilizing the monotonicity of $\Phi^*$, we can deduce from Lemma \ref{lem-lim} that 
\begin{equation}\label{eq-vphi}
  V(t,x)\geq \Phi^*(t,L_0-x) \text{ for }(t,x)\in\mathbb{R}\times(-\infty,L_0].
\end{equation}
Define
\begin{equation*}
  L^*:=\sup\{L\in\mathbb{R}:V(t,x)\geq \Phi^*(t,L-x) \text{ for }(t,x)\in\mathbb{R}\times(-\infty,L]\}.
\end{equation*}
From  \eqref{eq-vphi} we see that $L^*\in [L_0,\infty]$ is well-defined. Since $V(t,L(t))=\delta$ with $Q\leq L(t)\leq 3Q$ and $\Phi^*(t,x)>\delta$ for $t\in\mathbb R$, $x>0$, we must have $L^*\leq \Gamma(r_*)$. Thus we have 
\begin{equation}\label{eq-vpr}
 \begin{cases} V(t,x)\geq \Phi^*(t,L^*-x) \text{ for }(t,x)\in\mathbb{R}\times(-\infty,L^*],\\
 \min_{t\in\mathbb R}\Gamma(t)=\Gamma (r_*)\geq  L^*.
 \end{cases}
\end{equation}

  \begin{lemma}\label{lemm-gr}
$L^*=\Gamma(r_*)$.
\end{lemma}
\begin{proof}
Suppose to the contrary that $L^*<\Gamma(r_*)$. We are going to derive a contradiction in  several steps.

\textbf{ Step 1.} We prove that 
\begin{equation}\label{eq-v>p}
  V(t,x)> \Phi^*(t,L^*-x) ~~\text{ for } (t,x)\in\mathbb R \times(-\infty,L^*].
\end{equation}
Otherwise, in view of \eqref{eq-vpr}, there exists $(t_0,x_0)\in \mathbb R \times(-\infty,L^*]$ such that
\begin{equation*}
  V(t_0,x_0)= \Phi^*(t_0,L^*-x_0).
\end{equation*}
Since $L^*<\Gamma(r_*)\leq \Gamma(t)$, we obtain $V(t,L^*)>\delta=\Phi^*(t,0)$ for all $t\in\mathbb R$, which implies $x_0<L^*$. Thus $w(t,x):=V(t,x)-\Phi^*(t,L^*-x)$ satisfies $w\geq 0$ in $\mathbb R \times(-\infty,L^*]$ and
\begin{equation*}\begin{cases}
  w_t-dw_{xx}-k^*(t)w_x-c(t,x)w=0 &\text{ for }(t,x)\in\mathbb R \times(-\infty,L^*),\\
  w(t,L^*)> 0 &\text{ for }t\in\mathbb R,
\end{cases}
\end{equation*}
with 
$$c(t,x):=\int_0^1\partial_uf(t,s\Phi^*+(1-s)V)ds\ \  \mbox{ a bounded function.}$$
  By the parabolic strong maximum principle, we deduce  $w(t,x)>0$ for $(t,x)\in\mathbb R \times(-\infty,L^*]$, which contradicts  $w(t_0,x_0)=0$. Hence, \eqref{eq-v>p} holds.

 \textbf{ Step 2.} We show that for any $y<L^*$, 
 \begin{equation}\label{eq-chi}
  \chi (y):=\inf_{(t,x)\in\mathbb{R}\times[y,L^*]}[V(t,x)-\Phi^*(t,L^*-x)]>0.
 \end{equation}
Clearly, $\chi(y)\geq 0$ for $y<L^*$. Assume to the contrary that \eqref{eq-chi} does not hold,  then there exists  $y_0\in(-\infty,L^*)$ such that 
\begin{equation*}
  \chi(y_0)=0.
\end{equation*}
Hence, we can find a sequence $\{(s_n,x_n)\}\subset \mathbb{R}\times [y_0,L^*]$ such that 
\begin{equation*}
  \lim_{n\to\infty}[V(s_n,x_n)-\Phi^*(s_n,L^*-x_n)]=0.
\end{equation*}
In view of Step 1, we have $|s_n|\to\infty$. Write $s_n=\tilde{k}_nT+\tilde{r}_n$ with $\tilde{k}_n\in\mathbb Z$, $\tilde r_n\in [0, T)$. Up to extraction of a subsequence if necessary, we may assume that $\tilde{r}_n\to\tilde{r}\in[0,T]$ and $x_n\to x_0\in[y_0,L^*]$. Set 
\begin{equation*}
  (V_n(t,x), \Gamma_n(t)):=(V(t+s_n-\tilde{r},x+x_n),\Gamma(t+s_n-\tilde{r})-x_n).
\end{equation*} 
 By applying the standard parabolic estimates and up to extraction of a subsequence,  we can argue similarly  to  Lemma \ref{lem-lim} to conclude that 
 \begin{equation*}
  (V_n(t,x), \Gamma_n(t))\to (V^\infty,\Gamma^\infty) \text{ in }C_{loc}^{(1+\alpha)/2,1+\alpha}(\overline{\Sigma}_\infty)\times C_{loc}^{1+\alpha/2}(\mathbb{R}),
 \end{equation*}
 where ${\Sigma}_\infty:=\{(t,x):t\in\mathbb{R},x<\Gamma^\infty(t)\}$ and $(V^\infty,\Gamma^\infty)$ satisfies
\begin{equation}\label{eq-Vinf}\begin{cases}
  V^\infty_t-dV^\infty_{xx}-k^*(t)V^\infty_x=f(t,V^\infty), V^\infty>\delta &\text{ for }(t,x)\in{\Sigma}_\infty,\\
  V^\infty(t,\Gamma^\infty(t))=\delta &\text{ for }t\in\mathbb R
\end{cases}
\end{equation}
and 
\begin{equation}\label{eq-Vin}
  V^\infty(\tilde{r},0)=\Phi^*(\tilde{r}, L^*-x_0).
\end{equation}
Moreover, $\Gamma_n(t)=\Gamma(t+s_n-\tilde{r})-x_n\geq \Gamma(r_*)-x_n $ leads to
 $\Gamma^\infty(t)+x_0\geq \Gamma(r_*)>L^*$ and \eqref{eq-v>p} yields
\begin{equation*}
  V^\infty(t,x)\geq \Phi^*(t,L^*-x_0-x)~~ \text{ for }(t,x)\in\mathbb{R} \times (-\infty, L^*-x_0].
\end{equation*}
In addition, it is easily seen that $\Phi^*(t,L^*-x_0-x)$ solves \eqref{eq-Vinf} with $\Gamma^\infty(t)$ replaced by $L^*-x_0$. Due to $L^\infty(t)>L^*-x_0>0$, we have
\begin{equation*}
  V^\infty(t,L^*-x_0)>\delta=\Phi^*(t,0) ~~\text{ for }t\in\mathbb{R}.
\end{equation*} 
Hence, we can apply the parabolic strong maximum principle to compare $V^\infty(t,x)$ with $\Phi^*(t,L^*-x_0-x)$ over $(-\infty,0]\times (-\infty,L^*-x_0)$ to conclude that 
\begin{equation*}
  V^\infty(t,x) >\Phi^*(t,L^*-x_0-x) \text{ for }t\leq \tilde{r},\ x<L^*-x_0,
\end{equation*}
which contradicts \eqref{eq-Vin}. Therefore, \eqref{eq-chi} holds for all $y< L^*$.

\textbf{Step 3.} We find a constant $L_0<L^*$ such that  $V(t,x)\geq \Phi^*(t,L^*-x+\epsilon)$ in $\mathbb{R}\times(-\infty,L_0]$ for all small $\epsilon>0$.

Since $f_u(\cdot,1)<0$ and $f$ is $C^1$ in $u$ uniformly for $t\in\mathbb R$,  there exists $\epsilon_0>0$ small enough such that  
\begin{equation}\label{eq-f_u}
  f_u(t,u)<0 ~~\text{ for }t\in\mathbb{R},\ u\in[1-\epsilon_0,1].
\end{equation} 
Owing to $\Phi^*(t,x) \to\infty$ as $x\to\infty$ uniformly in $t\in\mathbb R$,   we can find $L_0=L_0(\epsilon_0)<L^*$ with $L^*-L_0\gg 1$  such that 
\begin{equation}\label{eq-pl1}
  \Phi^*(t,L^*-x)\geq 1-\epsilon_0 ~~\text{ for }t\in\mathbb{R},\ x\leq L_0.
\end{equation}
By continuity,  there exists  $\epsilon\in(0,\epsilon_0)$ small enough (depending on $L^*-L_0$) such that
\begin{equation}\label{eq-pchi}
  \Phi^*(t,L^*-L_0+\epsilon)\leq \Phi^*(t,L^*-L_0)+\chi(L_0)~ \text{ for }t\in\mathbb R,
\end{equation}
where $\chi$ is defined by \eqref{eq-chi}.
Fix any such  $\epsilon$. For any $L<L_0$, consider the following auxiliary problem 
\begin{equation}\label{eq-wR}
  \begin{cases}
    w_t-dw_{xx}-k^*(t)w_x=f(t,w), &t>0, x\in(L,L_0),\\
    w(t,L)=1, ~w(t,L_0)=\Phi^*(t,L^*-L_0+\epsilon), &t>0,\\
    w(0,x)=1-\epsilon_0, &x\in[L,L_0].
  \end{cases}
\end{equation}
Clearly, $1$ and $1-\epsilon_0$ form a pair of upper and lower solutions of \eqref{eq-wR}.  It follows from the standard upper and lower solution arguments that \eqref{eq-wR} admits a maximal solution 
$$w_L\in C_{loc}^{1+\alpha/2,2+\alpha}((0,\infty)\times[L,L_0])\cap C_{loc}^{1+\alpha/2,2+\alpha}([0,\infty)\times(L,L_0))$$
 satisfying
\begin{equation}\label{eq-1pw}
 1-\epsilon_0\leq  w_L(t,x)\leq 1 ~~\text{ for }t>0,\ x\in[L,L_0].
\end{equation}
By applying the usual comparison principle, we know that $w_L$ increases in $L$ in the sense that $L_2<L_1$ implies $w_{L_2}(t,x)\leq w_{L_1}(t,x)$. Hence, the pointwise limit
\begin{equation*}
   w_\infty(t,x):=\lim_{L\to\infty}w_L(t,x) \ \mbox{ for } t>0,x\in(-\infty,L_0]
 \end{equation*}
 is well-defined and  $1-\epsilon_0\leq w_\infty\leq 1$. By applying  the standard parabolic estimates, we conclude that the above convergence also holds in $C_{loc}^{1,2}([0,\infty)\times(-\infty,L_0))$,  and thus $w_\infty$ solves \eqref{eq-wR} in $[0,\infty) \times(-\infty,L_0]$  except for the first equation in the second line.
Since $1-\epsilon_0$ is a  lower solution and $f$  is $T$-periodic in $t$,  $w_\infty(t+kT,x)$ is nondecreasing  in $k\in\mathbb R$. Hence, the pointwise limit
\begin{equation*}
  \tilde{w}(t,x)=\lim_{k\to\infty}w_\infty(t+kT,x)
\end{equation*}
is well-defined and $\tilde{w}(t,x)$ is $T$-periodic  in $t$ with $1-\epsilon_0\leq \tilde{w}\leq 1$.   Similar to before,  we deduce that  $\tilde{w}$ solves
\begin{equation}\label{eq-wtild}
  \begin{cases}
    \tilde{w}_t-d\tilde{w}_{xx}-k^*(t)\tilde{w}_x=f(t,\tilde{w}), ~1-\epsilon_0\leq \tilde{w}\leq 1, &t\in\mathbb R,\  x\in(-\infty,L_0),\\
    \tilde{w}(t,L_0)=\Phi^*(t,L^*-L_0+\epsilon), &t\in\mathbb R,\\
    \tilde{w}(t,x)=\tilde{w}(t+T,x), &t\in\mathbb R,\  x\in(-\infty,L_0].
  \end{cases}
\end{equation}
We claim that $\tilde{w}(t,-\infty)=1$. Indeed, for any sequence $\{x_n\}$  satisfying $x_n\to-\infty$ as $n\to\infty$, if we  denote 
\[\tilde{w}_n(t,x):=\tilde{w}(t,x+x_n) \text{ for } t\in\mathbb{R}, x\in(-\infty, L_0-x_n],\]
then $\tilde{w}_n(t,x)$ satisfies \eqref{eq-wtild} with $L_0$ replaced by $L_0-x_n$. Hence, by using the standard parabolic estimates and up to extraction of a subsequence, we know that there exists a function $\hat{w}$ such that 
\[\tilde{w}_n(t,x)\to\hat{w}(t,x) \text{  as } n\to\infty \text{ locally uniformly in } C^{1, 2}(\mathbb{R}\times\mathbb{R}).\]
Moreover, $\hat{w}(t,x)$ satisfies 
\[ \begin{cases}
  \hat{w}_t-d\hat{w}_{xx}-k^*(t)\hat{w}_x=f(t,\hat{w}), ~1-\epsilon_0\leq \hat{w}\leq 1, &t\in\mathbb R, x\in\mathbb{R},\\
  \hat{w}(t,x)=\hat{w}(t+T,x), &t\in\mathbb R, x\in\mathbb{R}.
\end{cases}
\]
The usual comparison principle yields that $\hat{w}(t,x)\geq p(t)$ with $p(t)$ solving
\[p_t=f(t,p),~t>0;~~~p(0)=1-\epsilon_0.\]
Due to $f(\cdot,u)>0$ for $u\in(0,1)$, we know that $p(t+kT)\to 1$  locally uniformly in $t\in\mathbb{R}$ as the integer $ k\to\infty$. Therefore, 
\[1\geq \hat{w}(t,x)=\lim_{k\to\infty}\hat{w}(t+kT,x)\geq \lim_{k\to\infty}p(t+kT)=1  \text{ locally uniformly in }(t,x)\in\mathbb{R}\times\mathbb R,\]
which implies that 
\[\tilde{w}(t,x+x_n)\to 1   \text{ locally uniformly in } C^{1, 2}(\mathbb{R}\times\mathbb{R}).\]
By the arbitrariness of  $\{x_n\}$, we conclude that $\tilde{w}(t,-\infty)=1$. The claim is thus proved.

We now set 
$$\Psi(t,x):=\Phi^*(t,L^*-x+\epsilon) ~\text{ for }t\in\mathbb R, x\leq L_0.$$
Clearly, $\Psi$ also satisfies \eqref{eq-wtild}. Thanks to \eqref{eq-pl1}, one has $\Psi(t,x)\geq 1-\epsilon_0$ for $x\leq L_0$. Therefore, we can compare $\Psi$ and $w_\infty$  to obtain
\begin{equation*}
  \Psi(t+kT,x)\geq  w_\infty(t+kT,x)~~\text{ for } t>-kT,\  x\in(-\infty,L_0],\  k\in\mathbb N.
\end{equation*}
Now, letting $k\to\infty$, one gets
\begin{equation}\label{eq-wp1}
  \Psi(t,x)\geq \tilde{w}(t,x)\geq 1-\epsilon_0~~\text{ for }t\in\mathbb{R},\ x\in(-\infty,L_0].
\end{equation}

We claim that 
\begin{equation*}
  \Psi(t,x)\equiv \tilde{w}(t,x)~~ \text{ for }t\in\mathbb{R},\  x\in(-\infty,L_0].
\end{equation*}
Denote
\begin{equation*}
  U(t,x):=\Psi(t,x)-\tilde{w}(t,x).
\end{equation*}
Then $U\geq 0$ satisfies 
\begin{equation}\label{eq-U}\begin{cases}
  U_t-dU_{xx}-k^*(t)U_x-f_u(t,\theta(t,x))U=0,&t\in\mathbb{R},\  x\in(-\infty,L_0),\\
  U(t,-\infty)=U(t,L_0)=0,&t\in\mathbb{R},\\
  U(t,x)=U(t+T,x), &t\in\mathbb{R},\ x\in(-\infty, L_0],
\end{cases}
\end{equation}
where $\theta(t,x)\in[\tilde{w}(t,x),\Psi(t,x)]\subset[1-\epsilon_0,1]$.
If the claim does not hold, then in view of \eqref{eq-wp1} and $U(t,-\infty)=0$, there exists $(t_0,x_0)\in\mathbb{R}\times(-\infty,L_0)$ such that 
$$U(t_0,x_0)=\max_{(t,x)\in\mathbb{R}\times(-\infty,L_0)}U(t,x)>0.$$
Since  $f_u(\cdot,u)<0$ for $u\in[1-\epsilon_0,1]$, using the first equation of \eqref{eq-U} we obtain
\begin{equation*}
  0=\big[U_t-dU_{xx}-k^*(t)U_x-f_u(t,\theta(t,x))U\big]\big|_{(t_0,x_0)}\geq -f_u(t,\theta(t_0,x_0))U(t_0,x_0)>0.
\end{equation*}
This contradiction proves the claim.

By \eqref{eq-v>p},\eqref{eq-chi} and \eqref{eq-pchi}, we have
\begin{eqnarray*}
  &&V(t,x)> \Phi^*(t,L^*-x)\geq 1-\epsilon_0 ~~\text{ for } (t,x)\in\mathbb R \times(-\infty,L_0],\\
  && V(t,L_0)\geq \Phi^*(t,L^*-L_0)+\chi (L_0)\geq  \Phi^*(t,L^*-L_0+\epsilon),
\end{eqnarray*}
 and thus we can compare $V$ with $w_\infty$  to deduce that 
\[ V(t,x)\geq w_\infty(t+kT,x)~\text{ for }t>-kT,\ x\in(-\infty,L_0],\  k\in\mathbb N.\]
Letting $k\to\infty$ we obtain 
\begin{equation}\label{eq-vw}
  V(t,x)\geq \tilde{w}(t,x)=\Phi^*(t,L^*-x+\epsilon) ~\text{ for }t\in\mathbb R,\ x\in(-\infty,L_0].
\end{equation}

\textbf{Step 4.} Completion of the proof.

By Step 2, we know that
\begin{equation*}
  \varrho:=\chi(L_0)>0.
\end{equation*}
Since $L^*<\Gamma(r_*)\leq \Gamma(t)$ for all $t\in\mathbb R$, we deduce from the  interior parabolic estimates that 
\begin{equation*}
  |V_x(t,x)|\leq C ~~\text{ for } t\in\mathbb R,~ L^*\leq x\leq (\Gamma(r_*)+L^*)/2,
\end{equation*}
and thus, for any $\epsilon_1\in(0,\varrho ]$ small enough, we have
\begin{eqnarray*}
 V(t,x)>V(t,L^*)-\varrho/2 ~~\text{ for }t\in\mathbb R,\  x\in[L^*,L^*+\epsilon_1].
\end{eqnarray*}
Similarly,  one has
\[\Phi^{*}(t,L^*-x+\epsilon_1)\leq \Phi^*(t,L^*-x)+\varrho/2 ~~\text{ for }t\in\mathbb R,\ x\in[L_0, L^*].\]
Recall that 
\begin{equation*}
V(t,x)-\Phi^*(t,L^*-x)\geq \chi (L_0)=\varrho  \text{ for }(t,x)\in\mathbb{R}\times[L_0,L^*].
\end{equation*}
As a consequence, for $t\in\mathbb{R}$ and $x\in[L_0,L^*]$, we readily have 
\begin{equation*}
  V(t,x)-\Phi^*(t,L^*-x+\epsilon_1)\geq \Phi^*(t,L^*-x)+\varrho - \Phi^*(t,L^*-x)-\varrho/2\geq \varrho/2>0.
\end{equation*}
Up to decreasing $\epsilon_1$ if necessary,  we may also assume
\begin{equation*}
  \Phi^*(t,L^*-x+\epsilon_1)<\Phi^*(t,0)+\varrho/2 ~~\text{ for } x\in[L^*,L^*+\epsilon_1],\ t\in\mathbb R.
\end{equation*}
Hence,  for $t\in\mathbb{R}$ and $x\in[L^*,L^*+\epsilon_1]$, we have
\begin{equation*}
  V(t,x)-\Phi^*(t,L^*-x+\epsilon_1)> V(t,L^*)-\varrho /2 - \Phi^*(t,0)-\varrho/2\geq0.
\end{equation*}
These inequalities, together  with \eqref{eq-vw}, finally give 
\begin{equation*}
  V(t,x)-\Phi^*(t,L^*-x+\epsilon_1)\geq0 \text{ for } t\in\mathbb{R},\ x\in(-\infty,L^*+\epsilon_1)
\end{equation*}
for all small $\epsilon_1\in(0,\epsilon)$, which violates the definition of $L^*$. Therefore, $L^*=\Gamma(r_*)$ and the lemma is thus  proved.
\end{proof}

\begin{proposition}
  $V(t,x)\equiv \Phi^*(t,L^*-x)$ and $\Gamma(t)\equiv L^*$.
\end{proposition}

\begin{proof}
  By \eqref{eq-vpr} and  Lemma \ref{lemm-gr}, we know that $L^*=\Gamma(r_*)=\min_{t\in\mathbb R} \Gamma(t)$ and 
  \begin{equation*}
    V(t,x)\geq \Phi^*(t,L^*-x) \text{ for }t\in\mathbb{R},\ x\leq L^*
  \end{equation*}
  with 
  \begin{equation*}
    V(r_*,\Gamma(r_*))=\Phi^*(r_*,0)=\delta.
  \end{equation*}
 The parabolic strong maximum principle and Hopf boundary lemma then imply that
  \begin{equation*}
    \text{ either } V(t,x)\equiv \Phi^*(t,L^*-x) \text{ or } V_x(r_*,\Gamma(r_*))<-\Phi^*_x(r_*,0).
  \end{equation*} 
 By  \eqref{eq-V} and the minimality of $\Gamma(r_*)$, 
  \begin{equation*}
   0= \Gamma'(r_*)=-\frac{d}{\delta}V_x(r_*,\Gamma(r_*))-k^*(r_*),
  \end{equation*}
  and hence, 
  \begin{equation*}
    V_x(r_*,\Gamma(r_*))=-\frac{\delta}{d}k^*(r_*)=-\Phi_x^*(r_*,0).
  \end{equation*}
  Thus we necessarily have 
  \begin{equation}\label{eq-vp}
    V(t,x)\equiv \Phi^*(t,L^*-x), \text{ which implies }L(t)\equiv L^*. 
  \end{equation}
  The proof is complete.
\end{proof}
\subsection{Completion of the proof of Theorem \ref{thm-exact}}
We will finish the proof in three steps.\medskip

\textbf{Step 1.} Let $\{t_n\}$ be the time sequence used  in Lemma \ref{lem-lim}, with $r_n=t_n-k_nT\to r_*\in [0,T]$. We show that
\begin{equation*}
  \begin{cases}
  \lim_{n\to\infty}\big[h'(t+t_n-r_*)-k^*(t+t_n-r_*)\big]=0  \text{ for }t\in\mathbb R, \\
  \lim_{n\to\infty}\sup_{x\in[0,h(t_n)]}|u(t_n,x)-\Phi^*(t_n,h(t_n)-x)|=0.\\
  \end{cases}
\end{equation*}
By \eqref{eq-vp} and Lemmas \ref{lem-lim} and \ref{lemm-gr}, we have,  as $n\to\infty$,
\[\begin{cases}\rho_n(t)=h(t+t_n-r_*)-\int_0^{t+t_n-r_*}k^*(s)ds+2Q\to L^*& \text{ in } C_{loc}^{1+\alpha/2}(\mathbb R),\\
u(t+t_n-r_*,x+l(t+t_n-r_*))\to \Phi^*(t,L^*-x) &\text{ in } C_{loc}^{(1+\alpha)/2,1+\alpha}(\mathbb{R}\times(-\infty,0]).
\end{cases}\]
Hence,  
\[h'(t+t_n-r_*)-k^*(t+t_n-r_*)\to 0 \text{ in } C_{loc}^{\alpha/2}(\mathbb R) \text{ as }n\to\infty,\]
and for any $R>0$,
\begin{equation}\label{eq-uphi}
\begin{aligned}
  &\lim_{n\to\infty}\|u(t_n,\cdot)-\Phi^*(t_n,h(t_n)-\cdot)\|_{L^\infty( [h(t_n)-R,h(t_n)])}\\
  &=\lim_{n\to\infty}\|u(t_n,\cdot)-\Phi^*(r_*,h(t_n)-\cdot)\|_{L^\infty( [h(t_n)-R,h(t_n)])}=0.
  \end{aligned}
\end{equation}
 By  Lemmas \ref{lemma-upp} and \ref{lemma-lower},  for any $\epsilon>0$,  there exist  an integer $N>0$ and a constant  $R_1>0$ large enough such that 
\[1-\epsilon\leq  u(t_n,x)\leq 1+\epsilon ~\text{ for }x\in[0,h(t_n)-R_1],\ n\geq N.\]
By increasing $R_1$ if necessary, we also have
\[1-\epsilon\leq \Phi^*(t_n, h(t_n)-x)\leq 1+\epsilon,~ \forall n\in\mathbb N,\ x\in(-\infty,h(t_n)-R_1].\]
It follows that 
\[ \|u(t_n,\cdot)- \Phi^*(t_n, h(t_n)-\cdot)\|_{L^\infty([0,h(t_n)-R_1])}\leq 2\epsilon \mbox{  for $n\geq N$}.\]
 This, combined with \eqref{eq-uphi}, yields
\begin{equation}\label{eq-limup}
  \lim_{n\to\infty}\|u(t_n,\cdot)- \Phi^*(t_n, h(t_n)-\cdot)\|_{L^\infty([0,h(t_n)])}=0.
\end{equation}

 \textbf{Step 2.} We show that $\displaystyle\lim_{t\to\infty}\Big[h(t)-\int_0^tk^*(s)ds\Big]=h^*:=L^*-2Q$.

 From Step 1, we know that along the special sequence $\{t_n\}$ used  in Lemma \ref{lem-lim}, which  satisfies
 \[\lim_{n\to\infty}\left[h(t_n)-\int_0^{t_n}k^*(s)ds+2Q\right]=\liminf_{t\to\infty}\left[h(t)-\int_0^{t}k^*(s)ds+2Q\right]=L^*,\]
 we have   
\begin{equation}
  \lim_{n\to\infty}[h'(t_n)-k^*(t_n)]=0,
\end{equation}
and \eqref{eq-limup} holds.
Assume to the contrary that the desired  conclusion does not hold. 
Since 
\[\lim_{n\to\infty}\left[h(t_n)-\int_0^{t_n}k^*(s)ds\right]=\liminf_{t\to\infty}\left[h(t_n)-\int_0^{t}k^*(s)ds\right]=h^*,\]
 we can find a sequence $\{s_n\}$ such that $s_n\to \infty$ as $n\to\infty$ and 
\[\lim_{n\to\infty}\left[h(s_n)-\int_0^{s_n}k^*(s)ds\right]=\limsup_{t\to\infty}\left[h(t)-\int_0^{t}k^*(s)ds\right]=\hat{h}>h^*.\]
To derive a contradiction, we revisit the upper solution constructed in \eqref{eq-ups} and refine the parameters there by using \eqref{eq-limup}. 
Take 
\begin{equation}\label{parameters}
\mbox{$K_0=\gamma=(\hat{h}-h^*)/9$, $a\leq(\hat{h}-h^*)/9$ and $\tau^*=t_n$}
\end{equation}
 with $n\in\mathbb{N}$ to be determined later. Reviewing the proof of Lemma \ref{lemma-upp},  we can choose  $q_0\in(0,\min\{\eta,\frac{\delta-\delta_a}{\delta_a}\})$ small enough such that inequalities \eqref{eq-gq1}, \eqref{eq-gq2}, \eqref{eq-gq3} and \eqref{eq-gq4} hold, which implies that 
\begin{equation}\label{h-u}
\begin{cases}\bar{h}'(t)\geq -\frac{d}{\delta}\bar{u}_x(t,\bar{h}(t)) &\text{ for }t\geq t_n, \\
\bar{u}_t-d\bar{u}_{xx}-f(t,\bar{u})\geq 0 & \text{ for }t\geq t_n,\ x\in[g(t),\bar{h}(t)].
\end{cases}
 \end{equation}
Fix the above $q_0$. Now we show that $u(t_n,x)\leq \bar{u}(t_n,x)$ holds for  $x\in[g(t_n),h(t_n)]$ and all large $n$.
Since $\bar{h}(t_n)-x+z_\delta(t)\geq K_0>0$ for $x\in[g(t_n),h(t_n)]$, we have 
$$\bar{\Phi} (t_n,\bar{h}(t_n)-x+z_\delta(t))=\Phi^*(t_n,\bar{h}(t_n)-x+z_\delta(t_n))\geq \Phi^*(t_n,K_0)>\delta \text{ for }x\in[g(t_n),h(t_n)].$$
Thanks to \eqref{eq-limup},  there exists $n_1\in\mathbb{N}$ large enough such that for all $n\geq n_1$,
\begin{eqnarray*}
  u(t_n,x)- \Phi^*(t_n, h(t_n)-x)<q_0   \text{ for }x\in[0,h(t_n)].
\end{eqnarray*}
This, combined with the monotonicity of $\Phi^*$ and the fact that $\bar{h}(t_n)\geq h(t_n)+a\geq h(t_n)+z_\delta(t_n)$, implies that 
\begin{eqnarray*}
  u(t_n,x)< \Phi^*(t_n, \bar{h}(t_n)-x+z_\delta(t_n))+q_0 \text{ for }x\in[0,h(t_n)],\ n\geq n_1.
\end{eqnarray*}
Moreover, it is clear that  $\bar{h}(t_n)-x\geq h(t_n)+a$ for $g(t_n)\leq x\leq 0$. Due to $h(t)\to\infty$  as $t\to\infty$ and $\limsup_{t\to\infty}u(t,x)\leq 1$ uniformly in $x\in\mathbb{R}$  (by Proposition \ref{prop-summ}), we can increase $n_1$ if necessary to obtain 
\[ \Phi^*(t_{n},\bar{h}(t_{n})-x+z_\delta(t_{n}))+q_0\geq\Phi^*(t_{n},h(t_{n}))+q_0\geq  1+\frac{q_0}{2}\geq u(t_{n},x)\]
for $g(t_{n})\leq x\leq 0$ and $n\geq n_1$. 
As a result, we obtain 
$$u(t_{n},x)\leq \bar{u}(t_{n},x) \text{ for }x\in[g(t_n),h(t_n)],\ n\geq n_1.$$
It is evident from the definition of $\bar u$ that $\bar u(t, g(t))>\delta$. Hence recalling \eqref{h-u} we see that $(\bar{u},g, \bar{h})$ with $\tau^*=t_n$ is an upper solution of \eqref{free-bound} for  $t\geq t_{n}$, $n\geq n_1$, and by the comparison principle it follows that,  for such $t$ and $n$,
\[h(t)\leq \bar{h}(t)=\int_{t_{n}}^tk^*(s)ds+\gamma(1-e^{-\sigma_0(t-t_{n})})+h(t_{n})+a+K_0.\]
Hence, for all large $k$ satisfying $s_k>t_n$, we have
\[h(s_k)\leq h(t_{n})+\int^{s_k}_{t_{n}}k^*(s)ds+\gamma(1-e^{-\sigma_0(s_k-t_{n})})+a+K_0,\]
or equivalently,
\[h(s_k)-\int_0^{s_k}k^*(s)ds\leq h(t_{n})-\int_0^{t_{n}}k^*(s)ds+\gamma(1-e^{-\sigma_0(s_k-t_{n})})+a+K_0.\]
Letting $k\to\infty$ and then $n\to\infty$ in the above inequality, we obtain by recalling \eqref{parameters}, 
\[\hat{h}-h^*\leq \gamma+a+K_0\leq (\hat{h}-h^*)/3.\]
This is a contradiction, and Step 2 is completed.

\textbf{Step 3.}  We prove that $\lim_{t\to\infty} \sup_{x\in[0,h(t)]}|u(t,x)-\Phi^*(t,h(t)-x)|=0$.

By Step 2  for any given  sequence $\{t_n\}$ satisfying $t_n\to\infty$ as $n\to\infty$, we have
$$\lim_{n\to\infty}\rho(t_n)=\liminf_{t\to\infty}\rho(t).$$
This allows us to use Lemma \ref{lem-lim} to conclude that \eqref{eq-limup} holds for some subsequence of  $\{t_n\}$. The arbitrariness of  $\{t_n\}$ thus implies 
$$\lim_{t\to\infty} \sup_{x\in[0,h(t)]}|u(t,x)-\Phi^*(t,h(t)-x)|=0.$$

Symmetrically, we can prove 
 \begin{equation*}
  \lim_{t\to\infty}\sup_{x\in[g(t),0]}|u(t,x)-\Phi^*(t,x-g(t))|=0.
 \end{equation*}
 The proof   is now complete.\hfill $\Box$
\section*{Acknowledgments} The research of Du and Ma was supported by the Australian Research Council.   Wang's research was supported by NNSF of China (12471164,12071193).

\medskip

\noindent
{\bf Declarations:} \begin{enumerate}
\item There is no conflict of interest associated with this paper.
\item This paper has no associated data.

\end{enumerate}

\bibliographystyle{plain}

\end{document}